\documentclass[12pt,reqno]{amsart}

\usepackage[T1]{fontenc}
\usepackage[utf8]{inputenc}
\usepackage{lmodern}
\usepackage{microtype}
\usepackage{mathtools}
\usepackage{amssymb,amsfonts,amsthm}
\usepackage{enumitem}
\usepackage{aliascnt}
\usepackage[authoryear,round]{natbib}
\usepackage[section]{placeins}
\usepackage{tikz}
\usepackage{pgfplots}
\pgfplotsset{compat=1.18}
\usepackage[colorlinks=true,linkcolor=blue!60!black,citecolor=blue!60!black,urlcolor=blue!60!black]{hyperref}
\usepackage[nameinlink,capitalise,noabbrev]{cleveref}

\newtheorem{mainresult}{Theorem}

\newtheorem{theorem}{Theorem}[section]
\newaliascnt{proposition}{theorem}
\newtheorem{proposition}[proposition]{Proposition}
\aliascntresetthe{proposition}
\newaliascnt{lemma}{theorem}
\newtheorem{lemma}[lemma]{Lemma}
\aliascntresetthe{lemma}
\newaliascnt{corollary}{theorem}
\newtheorem{corollary}[corollary]{Corollary}
\aliascntresetthe{corollary}
\theoremstyle{definition}
\newaliascnt{definition}{theorem}
\newtheorem{definition}[definition]{Definition}
\aliascntresetthe{definition}
\newaliascnt{example}{theorem}

\aliascntresetthe{example}
\theoremstyle{remark}
\newaliascnt{remark}{theorem}
\newtheorem{remark}[remark]{Remark}
\aliascntresetthe{remark}
\crefname{mainresult}{Theorem}{Theorems}
\Crefname{mainresult}{Theorem}{Theorems}
\crefname{theorem}{Theorem}{Theorems}
\crefname{lemma}{Lemma}{Lemmas}
\crefname{proposition}{Proposition}{Propositions}
\crefname{corollary}{Corollary}{Corollaries}
\crefname{definition}{Definition}{Definitions}
\crefname{remark}{Remark}{Remarks}
\crefname{appendix}{Appendix}{Appendices}
\Crefname{appendix}{Appendix}{Appendices}

\newcommand{\DeltaJ}{\Delta^n}
\newcommand{\R}{\mathbb{R}}
\newcommand{\B}{z^*}
\newcommand{\Ximp}{X^I}
\newcommand{\argmax}{\operatorname*{arg\,max}}
\newcommand{\dist}{\operatorname{dist}}
\newcommand{\id}{\mathrm{id}}

\newcommand{\pure}[1]{\alpha(#1)}
\newcommand{\tailored}[1]{\alpha^I(#1)}
\newcommand{\pures}{\mathcal P}

\title[Mixed Worldviews and Chosen Preferences]{Persistent Mixed Worldviews in the theory of Chosen Preferences}
\author{Ziran Liu}

\address[Z. Liu]
{Shanghai Institute for Mathematics and Interdisciplinary Sciences (SIMIS), Shanghai, China, 200433}
\address[Z. Liu]
{Research Institute of Intelligent Complex Systems, Fudan University, Shanghai 200433, China}
\email{zliu@simis.cn}
\hypersetup{
  pdftitle={Persistent Mixed Worldviews in the theory of Chosen Preferences},
  pdfauthor={Ziran Liu},
  pdfsubject={Chosen preferences, stationary Markov-perfect equilibrium, and convergence},
  pdfkeywords={chosen preferences, worldview, mindset flexibility, tailored worldviews, Markov-perfect equilibrium}
}

\begin{document}
\begin{abstract}
\emergencystretch=1.5em
\citet*[\emph{American \mbox{Economic} Review}]{BBMZ2021} develop a theory of chosen preferences and conjecture convergence to pure worldviews at low mindset flexibility. We show that the (original) conjecture fails: with arbitrarily many worldviews, a common-value action sustains mixed worldviews in every stationary equilibrium from a relatively open set of initial worldviews, even when all pure worldviews are absorbing. 

However, finite linear programs characterize a sufficient persistence condition. Conversely, mixed accumulation points require shared valuation maxima. If each implementable action has a unique maximizing evaluator, every existing stationary equilibrium converges to purity for all interior flexibility and discount parameters. Experienced utility converges without uniqueness. In existing equilibria, small payoff perturbations preserve compromise over fixed horizons with high probability. Strictly tailored worldviews yield existence and a one-step equilibrium characterization at sufficiently low mindset flexibility; weak implementation can preclude stationary pure-strategy existence. The persistence and convergence results extend to stationary randomization.
\end{abstract}

\maketitle


\section{Introduction}\label{sec:introduction}

In \emph{A Theory of Chosen Preferences}, \citet*{BBMZ2021} model a consumer who chooses the worldview through which she evaluates her experiences. They conjecture that this process leads to pure worldviews at low mindset flexibility \citep[pp.~731--732, footnote~18]{BBMZ2021}. Choosing a worldview also changes which action the consumer will take. The current self chooses the next worldview but cannot bind later selves, and she evaluates future actions partly through her current perspective and partly through the perspective she expects to hold when they occur. A change in preferences can therefore make an experience more enjoyable while inducing an action that the current self would rather avoid. The question is whether this conflict can sustain a mixed worldview, despite the freedom to adopt a pure one.

Shared highest evaluations are necessary for a mixed worldview to remain a limit point of equilibrium choices. A common-value action can sustain such mixing indefinitely, even though adopting a pure worldview is feasible and costless. Conversely, every mixed accumulation point of an equilibrium path requires agreement at the highest evaluation of each action optimal there. If every implementable action has a unique maximizing pure evaluator, every existing stationary equilibrium converges to a pure worldview. The same comparison across worldviews thus explains the persistence examples and yields a general convergence result. Small differences in evaluation can restore eventual purity without producing a quick change in behavior.

We retain the finite action menu, linear mixing of worldview utilities, and sophisticated behavior of~\citet{BBMZ2021}. A higher weight $\lambda\in(0,1)$ on the current worldview means lower \emph{mindset flexibility}; $\delta\in(0,1)$ is the discount factor. We study stationary Markov-perfect equilibria, allowing randomization over actions and next-period worldviews. The convergence results apply to equilibria that exist. We establish existence in the persistence examples and under strictly tailored worldviews, and give a finite-menu example with no stationary pure-strategy equilibrium.

\citet[Propositions~1 and~3, pp.~728--731]{BBMZ2021} establish immediate or gradual adjustment to purity in their two-worldview, two-action analysis. The conjectured extension does not hold without further restrictions on finite action menus. With $n\geq2$ pure worldviews, consider a common action yielding zero and $n$ associated actions, each yielding $1$ under its own worldview and $-M$ under the others, where $M>1/(n-1)$. From the equal mixture, the exact threshold is
\[
 \lambda_c=\frac{n}{(n-1)(1+M)}.
\]
Above this threshold, every stationary equilibrium remains a positive distance from all pure worldviews; below it, every equilibrium adopts a pure worldview in one period. Both outcomes occur at equality. Each pure worldview is absorbing throughout. At high $\lambda$, persistence also holds from a nonempty relatively open region of initial worldviews. Persistence therefore coexists with absorption at every pure state and holds across equilibrium selections.

The common action is never optimal under a pure worldview, but it is optimal under a mixture. At low mindset flexibility, the current self's objection to the action induced by conversion outweighs the anticipated gain in experienced utility. We formulate this comparison as an upper bound on the assessment of all implementable future worldview--action pairs. Repeating the current worldview for one period supplies the reverse bound in equilibrium, forcing every future pair to attain equality. Finite linear programs over the static action regions characterize this sufficient persistence condition. The assessment bound holds throughout the specified initial region and restricts every stationary equilibrium, including equilibria that randomize over future worldviews.

A separate argument restricts every stationary equilibrium path. The option to postpone the prescribed continuation implies that normalized discounted experienced utility is nondecreasing along deterministic paths. Under randomization, its conditional-expectation counterpart is a bounded submartingale. Experienced utility converges, and fixed-horizon losses from future actions evaluated under the current worldview are summable. The optimal continuation value is also a supremum of affine functions of the current worldview, hence convex and Lipschitz even when policies are discontinuous. Together, these properties imply that every worldview in the support of an accumulation point maximizes the valuation of every action optimal there.

Unique actionwise valuation maxima consequently imply convergence to a single pure worldview for every $0<\lambda,\delta<1$, almost surely under randomization. The maximizing worldview need not implement the action: valuing an action more highly than any other worldview does not mean preferring it to all other actions. This condition is weaker than tailoring and is implied by the cross-worldview payoff distinctness used in the earlier working paper's small-$\lambda$ result \citep[Proposition~5, p.~14]{BBMZ2019}. The convergence question under that restriction is therefore answered affirmatively, conditional on stationary equilibrium existence. Nonpurification requires payoff ties; for fixed finite menus, the payoff arrays that admit it lie in a finite union of equality hyperplanes. Ties alone are not sufficient for persistence.

Exact payoff agreement and near agreement have different long-run implications. Fix parameters above the persistence threshold and perturb the entire payoff array by at most $\rho$, allowing the formerly common payoff to differ across worldviews. From the equal mixture, the probability of leaving the compromise action by a fixed date $T$ is bounded by a constant times $\rho\delta^{-(T-1)}$, uniformly over all existing stationary equilibria of the perturbed game. If the perturbation gives each action a unique maximizing evaluator, any such equilibrium eventually approaches a pure worldview. Yet the probability of retaining compromise throughout any prescribed finite horizon tends to one as $\rho$ tends to zero. The weights may change while the consumer retains the compromise action and remains a positive distance from purity. Eventual purity therefore need not entail an early observed change in behavior.

Tailored worldviews address a different question: whether a favorable evaluation can be reached together with a credible future action. Under a strict version of the tailoring assumptions in~\citet[Proposition~4, p.~732]{BBMZ2021}, each implementable action has a pure worldview that uniquely chooses it and values it more highly than every other worldview. At sufficiently low mindset flexibility, we prove existence and characterize all stationary equilibria by finite affine comparisons. Every consumer enters an absorbing pure worldview in one period. Randomization is possible only among maximizing pure destinations. Without full tailoring, unique valuation maxima and unique optimal actions at pure worldviews still give finite-time adoption at sufficiently low flexibility, conditional on existence. Weak tailoring gives asymptotic convergence whenever equilibrium exists, but tied actions can prevent a continuation optimum from being attained and destroy stationary pure-strategy existence.

Several approaches make preferences part of the economic problem rather than treating them as fixed primitives. \citet{AkerlofKranton2000} incorporate identity and prescriptions associated with social categories into utility. \citet{BeckerMulligan1997} study a consumer's effort to reduce the discount on future utility. \citet{PalaciosSantos2004} develop an incomplete-markets general-equilibrium model in which market arrangements and the formation of preferences interact. These papers provide antecedents for the choice of an evaluation rule; our analysis concerns the successive choices of that rule in the intrapersonal game introduced by \citet{BBMZ2021}.

Belief choice poses a related trade-off. \citet{BrunnermeierParker2005} balance the anticipatory benefit of optimism against the decisions it induces, while \citet{BenabouTirole2016} review research on motivated beliefs and the competing demands of accuracy and desirability. More directly concerned with changes in tastes, \citet{BoissonnetEtAl2023} give conditions under which changes in the relevance of attributes can be represented by maximization of a meta-preference. Their revealed-preference approach and the equilibrium approach here address different questions: we take the available worldview utilities as primitives and study the long-run consequences of sophisticated choice. Our convergence statements concern individual preference formation. Individual purity need not imply population polarization, and convergence of experienced utility along one equilibrium path is not a welfare ranking of policies that induce different preferences.

\Cref{sec:model} gives the model and main results. \Cref{sec:example,sec:persistence,sec:alignment} establish persistence and experienced-utility convergence. \Cref{sec:contactgeometry} derives the limiting-support restriction and generic purification. \Cref{sec:tailored,sec:weaktailoring} study tailoring and equilibrium existence. \Cref{sec:randomization} treats stationary randomization and near-common-value payoffs. \Cref{sec:implications} discusses the economic implications, and Appendix~\ref{sec:boundary} covers the boundary values of mindset flexibility.

\section{The model and main results}\label{sec:model}

\subsection{Worldviews and actions}

Let $J=\{1,\ldots,n\}$ and $X$ be nonempty finite sets of pure worldviews and actions, respectively.  Pure worldview $j$ evaluates action $x$ according to $u_j(x)\in\R$.  Following \citet[Section~I, p.~724]{BBMZ2021}, a mixed worldview is
\[
 \alpha=(\alpha^1,\ldots,\alpha^n)\in\Delta^n
 :=\left\{\alpha\in\R_+^n:\sum_{j=1}^n\alpha^j=1\right\},
\]
and its momentary utility is
\begin{equation}\label{eq:instantaneousU}
 U(\alpha,x)=\sum_{j=1}^n\alpha^j u_j(x).
\end{equation}
We use $\pure{j}$ for the pure worldview that puts weight one on $j$, and write
\[
 \pures=\{\pure{1},\ldots,\pure{n}\}.
\]
The static optimal-action correspondence and its value are
\begin{equation}\label{eq:hB}
 \B(\alpha)=\argmax_{x\in X}U(\alpha,x),
 \qquad h(\alpha)=\max_{x\in X}U(\alpha,x).
\end{equation}
$z^*(\alpha)$ is nonempty at every state, and $h$ is continuous and convex as the maximum of finitely many affine functions.  The set of implementable actions is
\begin{equation}\label{eq:implementable}
 \Ximp=\bigcup_{\alpha\in\Delta^n}\B(\alpha).
\end{equation}
An action outside $X^I$ is strictly below $h(\alpha)$ at every state and cannot be chosen in equilibrium: current actions do not affect future states or feasible sets. Deleting all such actions therefore leaves $h$, the optimal-action correspondence on the remaining menu, and the stationary equilibrium worldview policies unchanged.

Cross-worldview utility levels are part of the primitives.  A common positive affine transformation $\widehat u_j(x)=r u_j(x)+b$, where $r>0$ and $b$ do not depend on $j$ or $x$, preserves all choices.  In particular, the continuation value defined below transforms as
\[
 \widehat C_{\phi,z}(\beta;\alpha)
 =r C_{\phi,z}(\beta;\alpha)+\frac{b}{1-\delta}.
\]
Independent additive constants for different worldviews need not preserve the choice of a future worldview.  Consequently, the equality of the payoff to an action across worldviews is a substantive assumption, not a free normalization.  When a common payoff exists, a common translation can set that payoff to zero.

\subsection{Intertemporal evaluation}

Time is discrete, $t=0,1,\ldots$, and the action set is $X_t=X$ at every date.  The initial worldview $\alpha_0$ is given.  For $t\geq1$, $\alpha_t$ was selected in period $t-1$.  The date-$t$ self chooses $x_t$ and $\alpha_{t+1}$.  Fix
\[
 0<\delta<1,\qquad 0<\lambda<1.
\]
The discount factor is $\delta$.  Higher $\lambda$ means lower mindset flexibility.  For a continuation trajectory $\sigma_t=((\alpha_t,x_t),(\alpha_{t+1},x_{t+1}),\ldots)$, we use the infinite-horizon specialization of \citet[equation~(1), p.~725]{BBMZ2021}:
\begin{equation}\label{eq:BBMZutility}
 V_t(\sigma_t)=U(\alpha_t,x_t)+
 \sum_{k=1}^{\infty}\delta^k
 \left[(1-\lambda)U(\alpha_{t+k},x_{t+k})
       +\lambda U(\alpha_t,x_{t+k})\right].
\end{equation}
The two terms in brackets evaluate the same future action under, respectively, the future and current worldviews.  Current actions do not affect any subsequent choice set or transition.  Hence the current-action choice separates from the choice of the next worldview, and $x_t\in z^*(\alpha_t)$ in equilibrium.

\subsection{Stationary Markov-perfect equilibrium}

A stationary pure strategy is a pair of maps
\[
 \phi:\Delta^n\longrightarrow\Delta^n,
 \qquad z:\Delta^n\longrightarrow X.
\]
We write $\phi^0=\id$ and $\phi^k$ for its $k$-fold iterate.  For a current worldview $\alpha$ and a proposed next worldview $\beta$, define
\begin{equation}\label{eq:continuationV}
 \begin{split}
 C_{\phi,z}(\beta;\alpha)
 =\sum_{k=0}^{\infty}\delta^k\bigl[&
 (1-\lambda)U\bigl(\phi^k(\beta),z(\phi^k(\beta))\bigr)\\
 &+\lambda U\bigl(\alpha,z(\phi^k(\beta))\bigr)\bigr].
 \end{split}
\end{equation}
The continuation value is measured from the next period, before applying the common leading discount factor. The payoff from the current choices $(x,\beta)$ is
\begin{equation}\label{eq:currentpayoff}
 U(\alpha,x)+\delta C_{\phi,z}(\beta;\alpha).
\end{equation}
Let $B=\max_{j\in J,x\in X}|u_j(x)|$.  Every bracket in \eqref{eq:continuationV} has absolute value at most $B$, and its tail starting at $k=K$ is at most $B\delta^K/(1-\delta)$ in absolute value.  Thus the series converges absolutely, with a tail bound uniform in the current state, the successor state, and the strategy pair.  No continuity of the policy is required.

\begin{definition}[Stationary MPE]\label{def:MPE}
A stationary pure-strategy Markov-perfect equilibrium is a pair $(\phi,z)$ such that, for every $\alpha\in\Delta^n$,
\begin{align}
 z(\alpha)&\in z^*(\alpha),\label{eq:MPEaction}\\
 C_{\phi,z}(\phi(\alpha);\alpha)&\geq
 C_{\phi,z}(\beta;\alpha)
 \quad\text{for every }\beta\in\Delta^n.\label{eq:MPEworldview}
\end{align}
Unless a randomized equilibrium is explicitly specified, \emph{stationary MPE} refers to this pure-strategy concept.  Stationary randomization is defined in \cref{sec:randomization}.
\end{definition}

Each dated self chooses only its current action and the next worldview; it cannot bind subsequent selves.  Conditions \eqref{eq:MPEaction}--\eqref{eq:MPEworldview} impose optimality in every subgame against every choice available to that self, as in \citet[Section~II.A, pp.~727--728]{BBMZ2021}.  The general pathwise results apply to any stationary MPE that exists.  Existence is proved for the examples and for the strictly tailored class.  Resolving ties by a fixed ordering of the finite action and destination sets gives Borel measurable policies in these constructions. The deterministic pathwise results require no measurability assumption; a probability kernel is required for stationary randomization.

For an initial worldview $\alpha_0$, the induced path is
\begin{equation}\label{eq:path}
 \alpha_{t+1}=\phi(\alpha_t),\qquad x_t=z(\alpha_t).
\end{equation}
We distinguish convergence to a pure worldview,
\begin{equation}\label{eq:purelimit}
 \alpha_t\longrightarrow\pure{j}\quad\text{for some }j\in J,
\end{equation}
from the weaker condition
\begin{equation}\label{eq:asymptoticpurification}
 \dist(\alpha_t,\pures)\longrightarrow0.
\end{equation}
A worldview $\alpha$ is \emph{absorbing} under $\phi$ if $\phi(\alpha)=\alpha$.  Failure of \eqref{eq:asymptoticpurification} does not by itself rule out convergence to a mixed worldview.  Finite-time adoption of an absorbing pure worldview means
\begin{equation}\label{eq:finitepurification}
 \exists T\in\{0,1,\ldots\},\ j\in J:\quad
 \alpha_t=\pure{j}\quad\text{for every }t\geq T.
\end{equation}
Thus \eqref{eq:finitepurification} implies \eqref{eq:purelimit}, which implies \eqref{eq:asymptoticpurification}.  In the two-worldview examples we identify a state with its scalar weight on worldview~1, so the pure states are $0$ and $1$ and distances are the usual distances on $[0,1]$.  In the general simplex, $\dist_1(\alpha,A)=\inf_{a\in A}\|\alpha-a\|_1$, and an unsubscripted distance uses the same norm. None of the convergence statements depends on this choice. Distances in the scalar two-worldview parametrization are half the corresponding $\ell^1$ distances in the simplex.

\subsection{Main results}\label{sec:mainresults}

The persistence examples and sufficient condition are stated in \cref{thm:phasetransition,thm:manyworldviews,thm:generaltrap}. \Cref{thm:alignment,thm:genericpurification} give the general restrictions: experienced utility converges, and mixed accumulation points require shared valuation maxima. \Cref{thm:onestepcharacterization,thm:weaktailoringnonexistence} distinguish conditions giving existence and one-step adoption from conditions giving convergence only when an equilibrium exists. The proofs begin with deterministic strategies; \cref{sec:randomization} establishes the randomized conclusions, including those stated in \cref{thm:genericpurification}. \Cref{cor:nearcommon} then relates eventual purity in nearby models to the timing of departure from the compromise action.

\subsubsection*{A common-value action}
Consider $J=\{1,2\}$ and $X=\{c,a,b\}$, writing $\alpha\in[0,1]$ for the weight on worldview~1.  The payoffs are
\begin{equation}\label{eq:countermatrix}
\begin{array}{c|ccc}
 & c & a & b\\ \hline
 u_1 & 0 & 1 & -3\\
 u_2 & 0 & -3 & 1
\end{array}.
\end{equation}

\begin{mainresult}[A threshold for convergence]\label{thm:phasetransition}
For every $\delta\in(0,1)$ and $\lambda\in(0,1)$, the environment \eqref{eq:countermatrix} admits a stationary MPE.  From the initial worldview $\alpha_0=1/2$, the following conclusions hold:
\begin{enumerate}[label=(\roman*)]
\item If $0<\lambda<1/2$, then in every stationary MPE,
\[
 \alpha_1\in\{0,1\}
 \quad\text{and}\quad
 \alpha_t=\alpha_1\ \text{for all }t\geq1.
\]
Thus the consumer adopts a pure worldview after one period.
\item If $\lambda=1/2$, both stationary MPE with and without convergence to a pure worldview exist.  In every stationary MPE, every future worldview-action pair belongs to
\begin{equation}\label{eq:criticalequalityset}
 \left\{(\beta,c):\beta\in\left[\frac14,\frac34\right]\right\}
 \cup\{(1,a),(0,b)\}.
\end{equation}
\item If $1/2<\lambda<1$, then in every stationary MPE,
\begin{equation}\label{eq:nonpurifyall}
 x_t=c,
 \qquad
 \alpha_t\in\left[\frac14,\frac34\right]
 \qquad\text{for all }t\geq0.
\end{equation}
In particular,
\[
 \dist(\alpha_t,\{0,1\})\geq\frac14
 \qquad\text{for all }t,
\]
so even convergence to the set of pure worldviews is impossible.
\end{enumerate}
\end{mainresult}

Both pure worldviews uniquely select an extreme action, whereas $\alpha=1/2$ uniquely selects $c$.  Adopting the pure worldview associated with either extreme action yields experienced utility $1$, but the balanced worldview evaluates that action at $-1$.  The combined assessment is therefore $1-2\lambda$, compared with zero for $c$.

When $\lambda>1/2$, every stationary equilibrium continuation from $\alpha_0=1/2$ preserves the compromise action even though both pure worldviews are absorbing. Their absorption does not imply convergence from an initially mixed worldview.  The proof and global equilibrium constructions are in \cref{sec:example}.  \Cref{prop:asymmetric,prop:common-value-perturbations} extend the persistence conclusion to unequal pure-worldview peak utilities and to small perturbations that preserve the common-value action.

This restriction matters for the conjecture in \citet[pp.~731--732, footnote~18]{BBMZ2021}.  The earlier working paper's result for sufficiently small $\lambda$ assumes
\begin{equation}\label{eq:genericity}
 u_i(x)\neq u_j(y)
 \qquad\text{whenever }i\neq j,\quad x,y\in X.
\end{equation}
See \citet[Proposition~5, p.~14]{BBMZ2019}. Since $u_1(c)=u_2(c)$, the example violates \eqref{eq:genericity}. A separate positive result, \cref{thm:genericpurification}, establishes convergence under a weaker unique-maximizer condition, conditional on equilibrium existence.

\subsubsection*{Arbitrarily many pure worldviews}
\begin{mainresult}[Persistence with $n$ pure worldviews]\label{thm:manyworldviews}
Let $n\geq2$, $M>1/(n-1)$, $J=\{1,\ldots,n\}$, and $X=\{c,a_1,\ldots,a_n\}$.  Suppose
\begin{equation}\label{eq:nworldpayoffs}
 u_j(c)=0,\qquad
 u_j(a_i)=\begin{cases}1,&j=i,\\-M,&j\neq i.\end{cases}
\end{equation}
Put
\[
 \bar\alpha=(1/n,\ldots,1/n),\qquad
 \lambda_c=\frac{n}{(n-1)(1+M)}\in(0,1),\qquad
 C_c=\left\{\alpha\in\Delta^n:\alpha^i\leq\frac{M}{1+M}\ \forall i\right\}.
\]
A stationary MPE exists for every $\delta\in(0,1)$ and $\lambda\in(0,1)$.  Every pure worldview is absorbing in every such equilibrium.  Starting at $\bar\alpha$:
\begin{enumerate}[label=(\roman*)]
\item If $\lambda<\lambda_c$, every stationary MPE selects a pure worldview in one period and remains there.
\item If $\lambda>\lambda_c$, every stationary MPE satisfies
\begin{equation}\label{eq:nworldpersistence}
 x_t=c,\qquad \alpha_t\in C_c,\qquad
 \dist_1(\alpha_t,\pures)\geq\frac{2}{1+M}>0
 \quad\text{for all }t\geq0.
\end{equation}
The same conclusion holds from every initial state in the nonempty relatively open set
\begin{equation}\label{eq:nworldopenset}
 \mathcal T_\lambda=\left\{\alpha\in\Delta^n:
 \max_i\alpha^i<1-\frac1{\lambda(1+M)}\right\}.
\end{equation}
\item At $\lambda=\lambda_c$, stationary MPE with pure and mixed absorbing continuations from $\bar\alpha$ both exist.
\end{enumerate}
In particular, $n=3$ and $M=2$ give $\lambda_c=1/2$.
\end{mainresult}

The proof and equilibrium constructions are in \cref{sec:manyworldviews}.  The equal-mixture initial state has full support for every $n$.  Persistence in fact holds on a relatively open set of initial worldviews, and the distance bound is uniform over all stationary equilibria.

\subsubsection*{A general condition for persistence}
The example uses a bound on the current self's evaluation of every future outcome that can be implemented.  For an arbitrary finite menu, this bound gives the following result.

\begin{definition}[Assessment of a future outcome]\label{def:composite}
Fix a current worldview $\bar\alpha\in\DeltaJ$.  For $\beta\in\DeltaJ$ and $x\in X$, define
\begin{equation}\label{eq:qgeneral}
 q_{\bar\alpha}(\beta,x)
 =(1-\lambda)U(\beta,x)+\lambda U(\bar\alpha,x).
\end{equation}
A pair $(\beta,x)$ is statically implementable when $x\in\B(\beta)$. The associated equality set is
\begin{equation}\label{eq:equalityset}
 K(\bar\alpha)
 =\left\{
 \beta\in\DeltaJ:
 \exists x\in\B(\beta)
 \text{ with }
 q_{\bar\alpha}(\beta,x)=h(\bar\alpha)
 \right\}.
\end{equation}
\end{definition}

\begin{mainresult}[Persistence in an equality set]\label{thm:generaltrap}
Suppose $\bar\alpha\in\DeltaJ$ satisfies the upper-bound condition
\begin{equation}\label{eq:flowceiling}
 q_{\bar\alpha}(\beta,x)
 \leq h(\bar\alpha)
 \qquad
 \text{for every }\beta\in\DeltaJ
 \text{ and every }x\in\B(\beta).
\end{equation}
Let $(\phi,z)$ be any stationary MPE and let $\alpha_0=\bar\alpha$.  Then, for every $t\geq1$,
\begin{equation}\label{eq:trapconclusionpair}
 q_{\bar\alpha}(\alpha_t,x_t)=h(\bar\alpha)
\end{equation}
and
\begin{equation}\label{eq:trapconclusionstate}
 \alpha_t\in K(\bar\alpha).
\end{equation}
In particular, if
\[
 K(\bar\alpha)\cap\pures=\varnothing,
\]
then $\varepsilon:=\dist(K(\bar\alpha),\pures)>0$ and every such equilibrium path satisfies
\begin{equation}\label{eq:mainuniformseparation}
 \dist(\alpha_t,\pures)\geq\varepsilon
 \qquad\text{for every }t\geq1.
\end{equation}
Thus no stationary MPE starting from $\bar\alpha$ approaches the set of pure worldviews.
\end{mainresult}

The condition concerns all statically implementable pairs $(\beta,x)$, including those not visited on a given equilibrium path.  If it holds, repeating the current worldview for one period forces the equilibrium continuation to attain the bound at every date.  \Cref{sec:persistence} proves the result and establishes compactness of $K(\bar\alpha)$.  When $K(\bar\alpha)$ excludes all pure worldviews, the distance from every future state to $\pures$ has a common strictly positive lower bound.

\subsubsection*{Experienced utility and future actions}
Experienced utility can converge even when worldviews remain mixed. The next result applies along every stationary equilibrium path.

Fix a stationary MPE $(\phi,z)$ and an induced path $(\alpha_t,x_t)_{t\geq0}$.  Set
\begin{equation}\label{eq:ht}
 h_t=U(\alpha_t,x_t)=h(\alpha_t).
\end{equation}
Define normalized discounted experienced utility by
\begin{equation}\label{eq:Ht}
 H_t=(1-\delta)\sum_{k=0}^{\infty}\delta^k h_{t+k}.
\end{equation}
The normalized experienced component of the date-$t$ continuation is $H_{t+1}$.  Define the current-worldview evaluation of future actions by
\begin{equation}\label{eq:Qt}
 Q_t=(1-\delta)\sum_{k=1}^{\infty}\delta^{k-1}
 U(\alpha_t,x_{t+k}),
\end{equation}
and define the normalized equilibrium continuation payoff by
\begin{equation}\label{eq:At}
 A_t=(1-\lambda)H_{t+1}+\lambda Q_t
 =(1-\delta)C_{\phi,z}(\phi(\alpha_t);\alpha_t).
\end{equation}
Write $\underline u=\min_{j,x}u_j(x)$ and $\overline u=\max_{j,x}u_j(x)$.

\begin{mainresult}[Monotonicity and agreement over future actions]\label{thm:alignment}
Let $J$ and $X$ be finite, $\delta\in(0,1)$, and $\lambda\in(0,1)$.  Along every path induced by a stationary MPE:
\begin{enumerate}[label=(\roman*)]
\item
\begin{equation}\label{eq:Atgeht}
 A_t\geq h_t
 \qquad\text{for every }t\geq0;
\end{equation}
\item
\begin{equation}\label{eq:Hmonotone}
 H_{t+1}\geq H_t
 \qquad\text{for every }t\geq0;
\end{equation}
\item there is $L\in\R$ such that
\begin{equation}\label{eq:convergences}
 H_t\longrightarrow L,
 \qquad
 h_t\longrightarrow L,
 \qquad
 A_t\longrightarrow L,
 \qquad
 Q_t\longrightarrow L;
\end{equation}
\item for every integer $k\geq1$, the losses $h_t-U(\alpha_t,x_{t+k})$ are nonnegative and satisfy
\begin{equation}\label{eq:totalgap}
 \begin{split}
 \sum_{t=0}^{\infty}\bigl[h_t-U(\alpha_t,x_{t+k})\bigr]
 &\leq\frac{(1-\lambda)(L-H_0)}{\lambda(1-\delta)^2\delta^{k-1}}\\
 &\leq\frac{(1-\lambda)(\overline u-\underline u)}
      {\lambda(1-\delta)^2\delta^{k-1}}.
 \end{split}
\end{equation}
In particular, for every fixed $k\geq1$,
\begin{equation}\label{eq:alignmentlimit}
 h_t-U(\alpha_t,x_{t+k})\longrightarrow0.
\end{equation}
\end{enumerate}
\end{mainresult}

\Cref{sec:alignment} gives the proof and the implications for cluster points and periodic paths.  The monotone object is the discounted average $H_t$, not necessarily the momentary utility $h_t$.  The loss $h_t-U(\alpha_t,x_{t+k})$ compares the current optimal action with the action chosen $k$ periods later, using the same current worldview for both evaluations.  Thus the result concerns agreement over actions along one equilibrium path, without requiring convergence of the worldviews or ranking alternative policies by welfare.

\subsubsection*{The support of a limiting worldview}
The actionwise maximum is
\[
 M(x)=\max_{j\in J}u_j(x),\qquad x\in X.
\]
For a fixed stationary MPE let
\begin{equation}\label{eq:contactdefinition}
 S(\alpha)=C_{\phi,z}(\phi(\alpha);\alpha),\qquad
 \mathcal Z=\{\alpha\in\Delta^n:(1-\delta)S(\alpha)=h(\alpha)\}.
\end{equation}
The set $\mathcal Z$ consists of worldviews at which the optimal continuation value equals the value of perpetually receiving the current optimal payoff.  Although the equilibrium policy need not be continuous, its continuation value is continuous and convex.

\begin{mainresult}[Limiting worldviews and unique valuation maxima]\label{thm:genericpurification}
Let $J$ and $X$ be finite and $0<\lambda,\delta<1$.  Fix any stationary MPE.
\begin{enumerate}[label=(\roman*),leftmargin=*]
\item The set $\mathcal Z$ is nonempty and compact.  Every accumulation point of every equilibrium path belongs to $\mathcal Z$.
\item If $p\in\mathcal Z$ and $x\in z^*(p)$, then
\begin{equation}\label{eq:contactmaxima}
 U(p,x)=h(p)=M(x),\qquad
 p^i>0\ \Longrightarrow\ u_i(x)=M(x).
\end{equation}
\item Suppose each implementable action has a unique maximizing pure worldview:
\begin{equation}\label{eq:uniquevaluator}
 \left|\argmax_{j\in J}u_j(x)\right|=1
 \qquad\text{for every }x\in X^I.
\end{equation}
Then, from every initial state, there is $j\in J$ such that
\begin{equation}\label{eq:genericpurelimit}
 \alpha_t\longrightarrow\alpha(j).
\end{equation}
This conclusion holds for every $\lambda\in(0,1)$, not only for low mindset flexibility.  The maximizing worldview in \eqref{eq:uniquevaluator} is not required to implement the action.
\end{enumerate}
For a stationary randomized MPE, use $S$ from \eqref{eq:randomS} in the definition of $\mathcal Z$. The same conclusions hold: for each fixed initial state the pathwise assertions hold almost surely, and the limiting pure worldview may be random. The support restriction in part~(ii) holds at every point of that equilibrium's contact set.
\end{mainresult}

The theorem separates two kinds of payoff restrictions.  Persistent mixing is possible in the unrestricted model, but a mixed accumulation point requires ties at the \emph{highest} evaluation of each action optimal there.  Unique actionwise valuation maxima rule out persistence without requiring tailored worldviews.  Cross-worldview payoff distinctness in \eqref{eq:genericity} implies \eqref{eq:uniquevaluator}; the convergence question under that restriction is therefore answered affirmatively, conditional on existence of a stationary equilibrium.  Neither restriction alone guarantees a stationary pure-strategy equilibrium.

\Cref{sec:contactgeometry} proves the result by combining \cref{thm:alignment} with convexity of the continuation value.  It also proves finite-time adoption at sufficiently low mindset flexibility when pure worldviews have unique optimal actions, and gives an explicit sufficient threshold.  The randomized proof appears in \cref{sec:randomcontact}.

\subsubsection*{Strictly tailored worldviews}
A different action-menu structure gives one-step convergence to a pure worldview.  \citet[p.~732, conditions~(i)--(ii) and Proposition~4]{BBMZ2021} call worldviews \emph{tailored} when each implementable action has a pure worldview that implements it and values it more highly than any other pure worldview.  Their result concerns na\"ive consumers.  Here the consumer is sophisticated, and we require the implementing action to be uniquely optimal.

\begin{definition}[Strictly tailored worldviews]\label{def:strictlytailored}
For every $x\in X^I$, suppose there is an index $i(x)\in J$ such that
\begin{align}
 u_{i(x)}(x)&>u_{i(x)}(y)
       &&\text{for every }y\in X\setminus\{x\},\label{eq:anchorstatic}\\
 u_{i(x)}(x)&>u_j(x)
       &&\text{for every }j\in J\setminus\{i(x)\}.\label{eq:anchorcross}
\end{align}
Write $\tailored{x}=\pure{i(x)}$ for the corresponding tailored pure worldview.  We refer to \eqref{eq:anchorstatic}--\eqref{eq:anchorcross} as the \emph{strictly tailored condition}.
\end{definition}

Each index $i(x)$ is uniquely determined by \eqref{eq:anchorcross}.  The map $x\mapsto i(x)$ is injective.  Indeed, if $i(x)=i(y)=i$ for $x\neq y$, then \eqref{eq:anchorstatic} would give both $u_i(x)>u_i(y)$ and $u_i(y)>u_i(x)$.  In particular, the condition implies $|X^I|\leq n$.

For each pair of distinct implementable actions define
\begin{align}
 D_{xy}&=u_{i(x)}(x)-u_{i(x)}(y)>0,\label{eq:Dxy}\\
 R_{xy}&=u_{i(y)}(y)-u_{i(x)}(y)>0.\label{eq:Rxy}
\end{align}
The positivity of $R_{xy}$ follows from \eqref{eq:anchorcross} for action $y$, since $i(x)\neq i(y)$.  Define
\begin{equation}\label{eq:lambdastar}
 \lambda^\star=
 \max\left\{0,\max_{\substack{x,y\in X^I\\x\neq y}}
                  \left(1-\frac{D_{xy}}{R_{xy}}\right)\right\}.
\end{equation}
If $X^I$ contains one action, the inner maximum is omitted and $\lambda^\star=0$.  Otherwise all the ratios are positive and there are only finitely many, so
\begin{equation}\label{eq:lambdastarlt1}
 0\leq\lambda^\star<1.
\end{equation}

For $x\in X^I$ define
\begin{align}
 g_x(\alpha)&=(1-\lambda)u_{i(x)}(x)+\lambda U(\alpha,x),\label{eq:gxanchor}\\
 G(\alpha)&=\max_{x\in X^I}g_x(\alpha),\qquad
 \Gamma(\alpha)=\argmax_{x\in X^I}g_x(\alpha).\label{eq:Ganchor}
\end{align}
The value $g_x(\alpha)$ is the current self's assessment of a period spent at $\alpha^I(x)$ taking action $x$.  Because $X^I$ is finite and nonempty, $\Gamma(\alpha)$ is nonempty.

\begin{mainresult}[Characterization and one-step convergence]\label{thm:onestepcharacterization}
Assume the strictly tailored condition, $\lambda^\star<\lambda<1$, and $0<\delta<1$.  A stationary MPE exists.  A stationary strategy pair $(\phi,z)$ is an MPE if and only if, at every $\alpha\in\Delta^n$,
\begin{equation}\label{eq:fullcharacterization}
 z(\alpha)\in z^*(\alpha),\qquad
 \phi(\alpha)\in\{\tailored{x}:x\in\Gamma(\alpha)\}.
\end{equation}
Every such equilibrium satisfies
\begin{equation}\label{eq:equilibriumvalue}
 C_{\phi,z}(\phi(\alpha);\alpha)=\frac{G(\alpha)}{1-\delta}.
\end{equation}
In particular, for every initial worldview $\alpha_0$, there is $x\in\Gamma(\alpha_0)$ such that
\begin{equation}\label{eq:finalanchorpath}
 \alpha_t=\tailored{x},\qquad x_t=x
 \quad\text{for every }t\geq1.
\end{equation}
\end{mainresult}

The sufficient threshold $\lambda^\star$ makes each tailored pure worldview absorbing in every stationary MPE.  Because every implementable action then has a credible pure destination that maximizes its experienced utility, the current self can attain a period-by-period upper bound immediately.  \Cref{sec:tailored} proves absorption, constructs equilibrium, and derives the characterization.  \Cref{cor:policycells} shows that the equilibrium set is independent of $\delta$ and that the worldview policy is uniquely determined outside finitely many proper affine hyperplanes.

\subsubsection*{The role of unique implementation}
The unique-optimal-action requirement in \cref{def:strictlytailored} makes a chosen pure worldview a credible destination for a particular action.  Weak optimality alone does not have that implication.  To state the distinction precisely, replace~\eqref{eq:anchorstatic} by
\begin{equation}\label{eq:weaktailoring}
 x\in z^*(\alpha(i(x)))
 \quad\text{for every }x\in X^I,
\end{equation}
while retaining the strict cross-worldview comparison~\eqref{eq:anchorcross}.  This reads ``implements'' as weak static optimality.  It allows two actions to have the same tailored pure worldview.  At that state a pure action policy must still choose just one of them.

\begin{mainresult}[Weak tailoring need not ensure a stationary pure-strategy equilibrium]\label{thm:weaktailoringnonexistence}
Let $J=\{1,2,3\}$, $X=\{a,b\}$, and
\begin{equation}\label{eq:weakcountermatrix}
\begin{array}{c|cc}
 &a&b\\\hline
 u_1&1&1\\
 u_2&0&-1\\
 u_3&-2&1/2
\end{array}.
\end{equation}
Both actions satisfy~\eqref{eq:weaktailoring} and~\eqref{eq:anchorcross} with $i(a)=i(b)=1$.  Nevertheless, for every $0<\lambda<1$ and $0<\delta<1$, this environment has no stationary pure-strategy MPE.
\end{mainresult}

Both actions have the same unique maximizing worldview. The failure of existence arises because their tie at that worldview prevents some selves from attaining a best continuation. \Cref{sec:weaktailoring} proves this nonattainment result.

\subsection{Postponing a change in worldview}

Choosing the current worldview once more postpones the stationary continuation by one period.  The following identity expresses this deviation; this postponement comparison also appears in \citet[Online Appendix, proof of Proposition~2, p.~3]{BBMZ2021}.

\begin{lemma}[One-period postponement]\label{lem:postponement}
For a stationary MPE $(\phi,z)$ and any current worldview $\alpha$, let
\[
 S(\alpha)=C_{\phi,z}(\phi(\alpha);\alpha).
\]
Then
\begin{equation}\label{eq:postponeidentity}
 C_{\phi,z}(\alpha;\alpha)=h(\alpha)+\delta S(\alpha),
 \qquad S(\alpha)\geq\frac{h(\alpha)}{1-\delta}.
\end{equation}
\end{lemma}
\begin{proof}
If the next worldview is $\alpha$, the next action is $z(\alpha)$, so the first bracket in \eqref{eq:continuationV} is
\[
 (1-\lambda)U(\alpha,z(\alpha))+\lambda U(\alpha,z(\alpha))=h(\alpha).
\]
For $k\geq1$, $\phi^k(\alpha)=\phi^{k-1}(\phi(\alpha))$.  Reindexing the absolutely convergent tail gives
\begin{align}
 C_{\phi,z}(\alpha;\alpha)-h(\alpha)
 &=\sum_{k=1}^{\infty}\delta^k
   \bigl[(1-\lambda)U(\phi^k(\alpha),z(\phi^k(\alpha)))
          +\lambda U(\alpha,z(\phi^k(\alpha)))\bigr]\notag\\
 &=\delta\sum_{r=0}^{\infty}\delta^r
   \bigl[(1-\lambda)U(\phi^r(\phi(\alpha)),z(\phi^r(\phi(\alpha))))\notag\\
 &\hspace{42mm}+\lambda U(\alpha,z(\phi^r(\phi(\alpha))))\bigr]\notag\\
 &=\delta C_{\phi,z}(\phi(\alpha);\alpha)
 =\delta S(\alpha).\label{eq:postponementreindex}
\end{align}
Since choosing $\alpha$ is feasible, \eqref{eq:MPEworldview} now gives
\[
 S(\alpha)\geq h(\alpha)+\delta S(\alpha)
 \quad\Longrightarrow\quad
 (1-\delta)S(\alpha)\geq h(\alpha)
 \quad\Longrightarrow\quad
 S(\alpha)\geq\frac{h(\alpha)}{1-\delta}.\qedhere
\]
\end{proof}

\section{A three-action example}\label{sec:example}

\subsection{The environment}

Consider the payoff array \eqref{eq:countermatrix}, with $\alpha\in[0,1]$ denoting the weight on worldview~1.  Its momentary utility functions are
\begin{equation}\label{eq:counterlines}
 U(\alpha,c)=0,
 \qquad
 U(\alpha,a)=4\alpha-3,
 \qquad
 U(\alpha,b)=1-4\alpha.
\end{equation}
The upper envelope of these three utility functions determines the current action, as shown in \cref{fig:utilitylines}.  The static optimal-action correspondence is
\begin{equation}\label{eq:Bcounter}
\B(\alpha)=
\begin{cases}
\{b\},&0\leq\alpha<\frac14,\\
\{b,c\},&\alpha=\frac14,\\
\{c\},&\frac14<\alpha<\frac34,\\
\{c,a\},&\alpha=\frac34,\\
\{a\},&\frac34<\alpha\leq1.
\end{cases}
\end{equation}
\begin{samepage}
In particular, the balanced worldview $\bar\alpha=1/2$ uniquely chooses $c$, whereas the two pure worldviews uniquely choose different extreme actions:
\[
 \B(1/2)=\{c\},\qquad \B(0)=\{b\},\qquad \B(1)=\{a\}.
\]
\end{samepage}

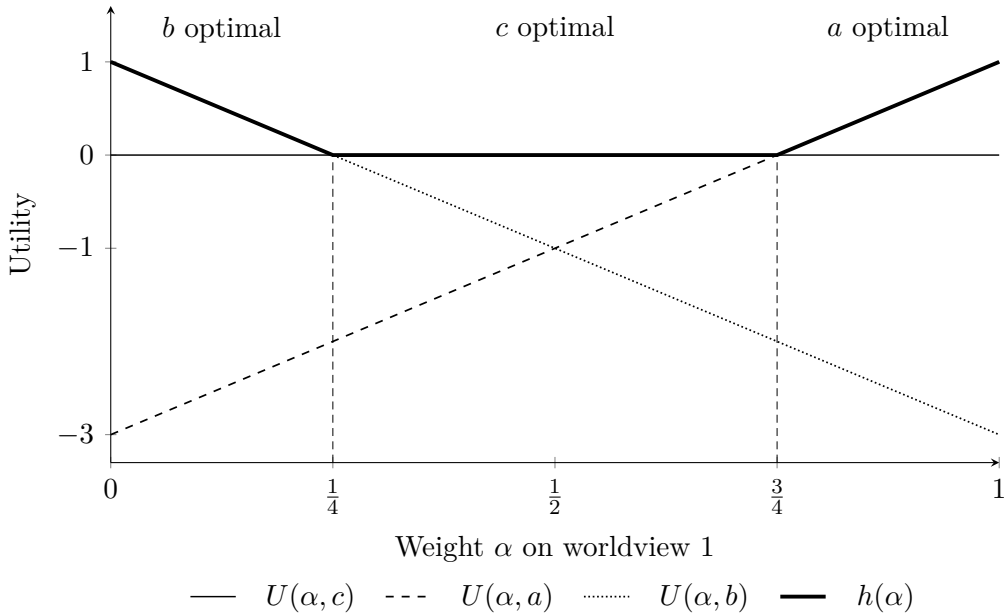
\begin{figure}[!htbp]
\centering
\begin{tikzpicture}
\begin{axis}[
 width=.84\textwidth,
 height=.48\textwidth,
 xmin=0,xmax=1,
 ymin=-3.3,ymax=1.6,
 axis lines=left,
 xlabel={Weight $\alpha$ on worldview 1},
 ylabel={Utility},
 xtick={0,0.25,0.5,0.75,1},
 xticklabels={$0$,$\frac14$,$\frac12$,$\frac34$,$1$},
 ytick={-3,-1,0,1},
 tick label style={font=\small},
 label style={font=\small},
 legend style={at={(0.5,-0.24)},anchor=north,legend columns=4,
               draw=none,font=\small,column sep=9pt},
 clip=false
]
\addplot[domain=0:1,samples=2,line width=.55pt] {0};
\addlegendentry{$U(\alpha,c)$}
\addplot[domain=0:1,samples=2,line width=.65pt,dashed] {4*x-3};
\addlegendentry{$U(\alpha,a)$}
\addplot[domain=0:1,samples=2,line width=.65pt,densely dotted] {1-4*x};
\addlegendentry{$U(\alpha,b)$}
\addplot[line width=1.45pt] coordinates {(0,1) (0.25,0) (0.75,0) (1,1)};
\addlegendentry{$h(\alpha)$}
\draw[densely dashed,line width=.35pt] (axis cs:0.25,-3.3) -- (axis cs:0.25,0);
\draw[densely dashed,line width=.35pt] (axis cs:0.75,-3.3) -- (axis cs:0.75,0);
\node[font=\small] at (axis cs:0.125,1.35) {$b$ optimal};
\node[font=\small] at (axis cs:0.5,1.35) {$c$ optimal};
\node[font=\small] at (axis cs:0.875,1.35) {$a$ optimal};
\end{axis}
\end{tikzpicture}
\caption{Current action choice in \eqref{eq:countermatrix}.  The bold upper envelope is $h(\alpha)=\max_x U(\alpha,x)$.  The compromise action $c$ is uniquely optimal for $1/4<\alpha<3/4$; at $1/4$ it ties with $b$, and at $3/4$ it ties with $a$.}
\label{fig:utilitylines}
\end{figure}

\begin{samepage}
The common value of $c$ can represent an outside option or a neutral status quo.  The balanced worldview evaluates either extreme action at $-1$, whereas the pure worldview that chooses it obtains utility $1$.  Adopting that pure worldview therefore yields the assessment $(1-\lambda)\cdot1+\lambda\cdot(-1)=1-2\lambda$, compared with zero for $c$.
\end{samepage}
\FloatBarrier

\subsection{Pure worldviews are absorbing}

Both pure worldviews are absorbing throughout $0<\lambda<1$.

\begin{lemma}\label{lem:pureabsorbcounter}
Fix $\delta\in(0,1)$ and $\lambda\in(0,1)$.  In every stationary MPE of the environment \eqref{eq:countermatrix},
\[
 \phi(1)=1,
 \qquad
 \phi(0)=0.
\]
\end{lemma}

\begin{proof}
The identities
\[
 U(1-\alpha,a)=U(\alpha,b),\qquad
 U(1-\alpha,b)=U(\alpha,a),\qquad U(\alpha,c)=0
\]
show that exchanging $a$ and $b$ and replacing $\alpha$ by $1-\alpha$ leaves the model unchanged.  It therefore suffices to prove absorption at worldview $1$.  Let $(\phi,z)$ be a stationary MPE and suppose, toward a contradiction, that
\[
 \beta_0:=\phi(1)\neq1.
\]
Starting from current worldview $1$, let
\[
 \beta_k=\phi^k(\beta_0),
 \qquad
 y_k=z(\beta_k),
 \qquad k\geq0.
\]
Define the date-$0$ one-period assessment
\[
 q_k=(1-\lambda)U(\beta_k,y_k)+\lambda U(1,y_k).
\]
For every $k$ we have $q_k\leq1$.  Indeed:
\begin{align*}
 y_k=a&\implies q_k
 =(1-\lambda)U(\beta_k,a)+\lambda
 \leq (1-\lambda)+\lambda=1,\\
 y_k=b&\implies q_k
 =(1-\lambda)U(\beta_k,b)-3\lambda
 \leq (1-\lambda)-3\lambda
 =1-4\lambda<1,\\
 y_k=c&\implies q_k=0<1.
\end{align*}
If $y_0=a$, then $\beta_0\neq1$ and \eqref{eq:counterlines} imply
\[
 U(\beta_0,a)<1,
\]
so $q_0<1$.  If $y_0\neq a$, the preceding display also gives $q_0<1$.  Hence
\begin{equation}\label{eq:Cpurestrict}
 C:=C_{\phi,z}(\beta_0;1)
 =\sum_{k=0}^{\infty}\delta^k q_k
 <\frac{1}{1-\delta}.
\end{equation}
If the current self deviates by choosing next worldview $1$, the next-period flow is exactly $1$, and from the following period onward the continuation is the original one.  The deviation payoff is therefore
\[
 1+\delta C.
\]
Equilibrium optimality requires
\[
 C\geq1+\delta C,
\]
or equivalently
\[
 C\geq\frac{1}{1-\delta},
\]
contradicting \eqref{eq:Cpurestrict}.  Thus $\phi(1)=1$.
\end{proof}

\subsection{The effect of mindset flexibility}

\begin{proof}[Proof of \cref{thm:phasetransition}]
Let $(\phi,z)$ be any stationary MPE.  Write
\[
 \alpha_1=\phi\left(\frac12\right),
 \qquad
 \alpha_{k+1}=\phi(\alpha_k),
 \qquad
 x_k=z(\alpha_k),
 \quad k\geq1.
\]
From the perspective of the balanced worldview, define the future one-period assessment
\begin{equation}\label{eq:qmid}
 q_k
 =(1-\lambda)U(\alpha_k,x_k)
 +\lambda U\left(\frac12,x_k\right),
 \qquad k\geq1,
\end{equation}
and the continuation payoff
\begin{equation}\label{eq:Cmid}
 C=\sum_{k=1}^{\infty}\delta^{k-1}q_k
 =C_{\phi,z}\left(\alpha_1;\frac12\right).
\end{equation}
For any future worldview $\beta$ and any action $x\in\B(\beta)$, \eqref{eq:counterlines} gives
\begin{align}
 x=c&\implies
 (1-\lambda)U(\beta,c)
 +\lambda U\left(\frac12,c\right)=0,
 \label{eq:qcmid}\\
 x=a&\implies
 (1-\lambda)U(\beta,a)
 +\lambda U\left(\frac12,a\right)
 \leq (1-\lambda)-\lambda=1-2\lambda,
 \label{eq:qamid}\\
 x=b&\implies
 (1-\lambda)U(\beta,b)
 +\lambda U\left(\frac12,b\right)
 \leq (1-\lambda)-\lambda=1-2\lambda.
 \label{eq:qbmid}
\end{align}
Equality in \eqref{eq:qamid} requires $\beta=1$ and $x=a$; equality in \eqref{eq:qbmid} requires $\beta=0$ and $x=b$.

\smallskip
\noindent\textit{Case 1: $0<\lambda<1/2$.}
Set
\[
 m_\lambda=1-2\lambda>0.
\]
Equations \eqref{eq:qcmid}--\eqref{eq:qbmid} imply
\[
 q_k\leq m_\lambda
 \quad\text{for every }k,
\]
and hence
\begin{equation}\label{eq:Cmidupperlow}
 C\leq\frac{m_\lambda}{1-\delta}.
\end{equation}
By \cref{lem:pureabsorbcounter}, choosing next worldview $1$ produces $(1,a)$ forever.  From the balanced current worldview, the corresponding constant one-period assessment is
\[
 (1-\lambda)U(1,a)
 +\lambda U\left(\frac12,a\right)
 =(1-\lambda)-\lambda
 =m_\lambda.
\]
Therefore equilibrium optimality implies
\begin{equation}\label{eq:Cmidlowerlow}
 C\geq\frac{m_\lambda}{1-\delta}.
\end{equation}
Combining \eqref{eq:Cmidupperlow} and \eqref{eq:Cmidlowerlow} gives
\[
 0=\frac{m_\lambda}{1-\delta}-C
   =\sum_{k=1}^{\infty}\delta^{k-1}(m_\lambda-q_k).
\]
Every summand is nonnegative, so $q_k=m_\lambda$ for every $k\geq1$.  In particular, $q_1=m_\lambda$.  The equality conditions following \eqref{eq:qbmid} imply
\[
 (\alpha_1,x_1)\in\{(1,a),(0,b)\}.
\]
Absorption follows from \cref{lem:pureabsorbcounter}.

\smallskip
\noindent\textit{Case 2: $\lambda=1/2$.}
Equations \eqref{eq:qcmid}--\eqref{eq:qbmid} give $q_k\leq0$.  If the current self chooses $1/2$ again, the next-period flow is zero, and thereafter the original continuation begins.  The deviation payoff is $\delta C$.  Equilibrium optimality yields
\[
 C\geq\delta C,
\]
so $C\geq0$.  Since $C\leq0$, we have $C=0$, and hence every $q_k=0$.  The equality cases in \eqref{eq:qcmid}--\eqref{eq:qbmid} are exactly the pairs listed in \eqref{eq:criticalequalityset}.  The existence of both stationary MPE with and without convergence to a pure worldview is established in \cref{prop:globalMPEcounter} below.

\smallskip
\noindent\textit{Case 3: $1/2<\lambda<1$.}
Now $1-2\lambda<0$.  Equations \eqref{eq:qcmid}--\eqref{eq:qbmid} imply $q_k\leq0$, with equality possible only when $x_k=c$.  The same delay deviation as in Case 2 yields $C\geq\delta C$, hence $C\geq0$.  Therefore $C=0$, every $q_k=0$, and every future action is $c$.  By \eqref{eq:Bcounter}, this forces
\[
 \alpha_k\in\left[\frac14,\frac34\right]
 \quad\text{for every }k\geq1.
\]
The same statements hold at $t=0$ because $\alpha_0=1/2$ and $x_0=c$.  The global constructions in \cref{prop:globalMPEcounter} establish equilibrium existence for every parameter pair, including both outcomes at $\lambda=1/2$.
\end{proof}

The following construction establishes existence for every parameter value in the theorem and supplies both kinds of equilibrium at the threshold.

\begin{proposition}[Explicit global stationary equilibria]\label{prop:globalMPEcounter}
For every $\delta\in(0,1)$ and $\lambda\in(0,1)$, the environment \eqref{eq:countermatrix} admits a stationary pure-strategy MPE.

Define three target states
\[
 p_c=\frac12,
 \qquad
 p_a=1,
 \qquad
 p_b=0,
\]
and, for a current worldview $\alpha$, define
\begin{align}
 g_c(\alpha)&=0,\label{eq:gc}\\
 g_a(\alpha)&=(1-\lambda)+\lambda(4\alpha-3),\label{eq:ga}\\
 g_b(\alpha)&=(1-\lambda)+\lambda(1-4\alpha).
 \label{eq:gb}
\end{align}
Let
\[
 G(\alpha)=\max\{g_c(\alpha),g_a(\alpha),g_b(\alpha)\},
\]
choose
\[
 r(\alpha)\in\argmax_{r\in\{c,a,b\}}g_r(\alpha),
\]
and set
\begin{equation}\label{eq:counterpolicy}
 \phi(\alpha)=p_{r(\alpha)},
 \qquad
 z(\alpha)\in\B(\alpha).
\end{equation}
At $\lambda=1/2$ choose the tie-breaking rule at $\alpha=1/2$ either as $r(1/2)=c$ or as one of the two extreme actions.  Each such choice produces a stationary MPE.  Under the first choice the consumer remains at the mixed worldview $1/2$; under either of the latter choices she adopts a pure worldview after one period.
\end{proposition}

\begin{proof}
We first verify that every state selected by the policy is absorbing.  At $p_a=1$,
\[
 g_a(1)=1,
 \qquad
 g_c(1)=0,
 \qquad
 g_b(1)=1-4\lambda<1,
\]
so $\phi(1)=1$.  At the other endpoint,
\[
 g_b(0)=1>\max\{0,1-4\lambda\}=\max\{g_c(0),g_a(0)\},
\]
so $\phi(0)=0$ as well.  If $r(\alpha)=c$ for some $\alpha$, then necessarily $\lambda\geq1/2$.  For $\lambda>1/2$,
\[
 g_c\left(\frac12\right)=0
 >1-2\lambda
 =g_a\left(\frac12\right)
 =g_b\left(\frac12\right),
\]
so $\phi(p_c)=p_c$.  At $\lambda=1/2$,
\[
 g_a(\alpha)=2\alpha-1,\qquad g_b(\alpha)=1-2\alpha,
 \qquad G(\alpha)=|2\alpha-1|.
\]
Thus $c$ is a maximizer only at $\alpha=1/2$.  If $r(1/2)=c$, then $p_c$ is absorbing.  If $r(1/2)=a$ or $b$, no state selects $p_c$, and the range of $\phi$ consists only of the absorbing endpoints.  If $\lambda<1/2$, $g_c$ is never maximal because
\[
 \max\{g_a(\alpha),g_b(\alpha)\}
 \geq \frac{g_a(\alpha)+g_b(\alpha)}{2}
 =1-2\lambda>0=g_c(\alpha).
\]
Thus $p_c$ is never selected in that case.

Fix a current worldview $\alpha$ and an arbitrary candidate next worldview $\beta$.  Let $y=z(\beta)$.  The first continuation flow satisfies
\begin{equation}\label{eq:firstflowG}
 (1-\lambda)U(\beta,y)+\lambda U(\alpha,y)
 \leq g_y(\alpha)
 \leq G(\alpha).
\end{equation}
Indeed, the first inequality is equality for $y=c$; for $y=a$ it follows from $U(\beta,a)\leq1$; and for $y=b$ it follows from $U(\beta,b)\leq1$.

From the next period onward, the policy sends $\beta$ to the absorbing target $p_{r(\beta)}$.  Every later flow, still evaluated by the original worldview $\alpha$, is
\[
 g_{r(\beta)}(\alpha)\leq G(\alpha).
\]
Consequently
\begin{equation}\label{eq:VupperG}
 C_{\phi,z}(\beta;\alpha)
 \leq\frac{G(\alpha)}{1-\delta}.
\end{equation}
Choosing
\[
 \beta=p_{r(\alpha)}
\]
produces the constant flow $G(\alpha)$ starting immediately, and therefore attains the upper bound in \eqref{eq:VupperG}.  Thus \eqref{eq:MPEworldview} holds for every $\alpha$, while \eqref{eq:MPEaction} holds by construction.
\end{proof}

For $\lambda>1/2$, the construction has a particularly simple form if ties among destination scores are resolved in favor of $p_c=1/2$.  Define
\[
 \ell_\lambda=\frac{1}{4\lambda},\qquad
 r_\lambda=1-\frac{1}{4\lambda}.
\]
Since
\[
 g_b(\alpha)\leq0\iff\alpha\geq\ell_\lambda,
 \qquad
 g_a(\alpha)\leq0\iff\alpha\leq r_\lambda,
\]
we have $g_c(\alpha)=G(\alpha)$ precisely on $[\ell_\lambda,r_\lambda]$.  To the left, $g_b>0>g_a$; to the right, $g_a>0>g_b$.  Thus the selected equilibrium policy is
\begin{equation}\label{eq:illustratedpolicy}
 \phi(\alpha)=
 \begin{cases}
 0,&0\leq\alpha<\ell_\lambda,\\
 \tfrac12,&\ell_\lambda\leq\alpha\leq r_\lambda,\\
 1,&r_\lambda<\alpha\leq1.
 \end{cases}
\end{equation}
For $1/2<\lambda<1$,
\[
 \frac14<\ell_\lambda<\frac12<r_\lambda<\frac34.
\]
The interval of initial worldviews sent to $1/2$ by this policy is therefore strictly smaller than the static region where $c$ is optimal.  \Cref{fig:equilibriumpolicy} shows the policy at $\lambda=2/3$.  Its intersections with the diagonal give the three absorbing worldviews $0$, $1/2$, and $1$.

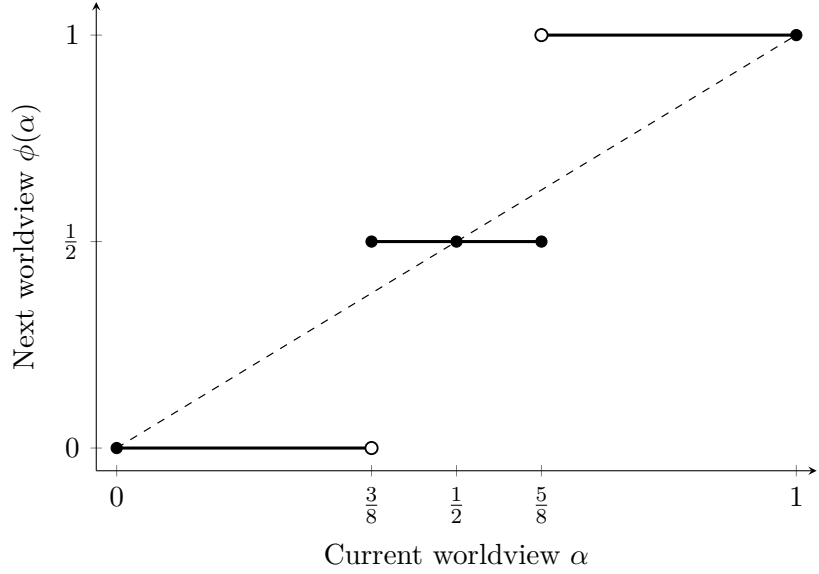
\begin{figure}[t]
\centering
\begin{tikzpicture}
\begin{axis}[
 width=.70\textwidth,
 height=.49\textwidth,
 xmin=-0.03,xmax=1.03,
 ymin=-0.055,ymax=1.08,
 axis lines=left,
 xlabel={Current worldview $\alpha$},
 ylabel={Next worldview $\phi(\alpha)$},
 xtick={0,0.375,0.5,0.625,1},
 xticklabels={$0$,$\frac38$,$\frac12$,$\frac58$,$1$},
 ytick={0,0.5,1},
 yticklabels={$0$,$\frac12$,$1$},
 tick label style={font=\small},
 label style={font=\small},
 clip=false
]
\addplot[domain=0:1,samples=2,dashed,line width=.45pt] {x};
\addplot[line width=1.15pt] coordinates {(0,0) (0.375,0)};
\addplot[line width=1.15pt] coordinates {(0.375,0.5) (0.625,0.5)};
\addplot[line width=1.15pt] coordinates {(0.625,1) (1,1)};
\addplot[only marks,mark=*,mark size=2.3pt,
 mark options={solid,fill=white,line width=.7pt}]
 coordinates {(0.375,0) (0.625,1)};
\addplot[only marks,mark=*,mark size=2.1pt,
 mark options={solid,fill=black}]
 coordinates {(0,0) (0.375,0.5) (0.5,0.5) (0.625,0.5) (1,1)};
\end{axis}
\end{tikzpicture}
\caption{The equilibrium selection \eqref{eq:illustratedpolicy} in the symmetric example at $\lambda=2/3$, for any $0<\delta<1$, with ties among destination scores assigned to the mixed destination $1/2$. Initial worldviews in $[3/8,5/8]$ reach the absorbing mixed worldview $1/2$ in one period; those outside this interval reach a pure worldview.  Filled and open circles give the included and excluded endpoint values.  The dashed line is the identity.}
\label{fig:equilibriumpolicy}
\end{figure}
\FloatBarrier

\subsection{Any finite number of pure worldviews}\label{sec:manyworldviews}

The same mechanism operates with arbitrarily many pure worldviews.  The following family also makes the dependence on the number of perspectives and the disagreement over actions explicit.

\begin{proof}[Proof of \cref{thm:manyworldviews}]
Linearity gives
\begin{equation}\label{eq:nworldutility}
 U(\alpha,a_i)=(1+M)\alpha^i-M,\qquad
 U(\bar\alpha,a_i)=\frac{1-(n-1)M}{n}<0.
\end{equation}
Consequently $c$ is uniquely optimal at $\bar\alpha$ and is optimal at $\alpha$ exactly when $\alpha\in C_c$.

Fix $i$ and evaluate any future pair $(\beta,y)$ from the pure worldview $\pure{i}$.  For $y=a_i$ the assessment is at most $1$, with equality only at $\beta=\pure{i}$.  For $y=a_j$, $j\neq i$, it is at most $1-\lambda(1+M)<1$; for $y=c$ it is zero.  Thus every continuation from $\pure{i}$ is bounded by $1/(1-\delta)$.  The postponement inequality gives the reverse bound in equilibrium.  Equality of the discounted sum forces equality in its first term, so $\phi(\pure{i})=\pure{i}$ and $z(\pure{i})=a_i$.

From $\bar\alpha$, every future extreme action has assessment at most
\begin{equation}\label{eq:nworldassessment}
 (1-\lambda)+\lambda\frac{1-(n-1)M}{n}
 =1-\frac{\lambda}{\lambda_c}=:q_\lambda,
\end{equation}
with equality only for $(\pure{i},a_i)$.  The assessment of $c$ is zero.  If $q_\lambda>0$, a pure absorbing destination attains the uniform bound $q_\lambda/(1-\delta)$; equality in the first-period bound forces immediate adoption of a pure worldview.  If $q_\lambda<0$, every continuation from $\bar\alpha$ is nonpositive, while postponement gives a nonnegative equilibrium continuation value because $h(\bar\alpha)=0$.  Every future assessment must therefore be zero, which forces $x_t=c$ and $\alpha_t\in C_c$.  For any $\alpha\in C_c$,
\[
 \|\alpha-\pure{i}\|_1
 =1-\alpha^i+\sum_{j\neq i}\alpha^j
 =2(1-\alpha^i)\geq\frac{2}{1+M}.
\]
This proves the distance bound.  If $s\in\mathcal T_\lambda$, then, for every $i$,
\[
 U(s,a_i)<1-\frac1\lambda<0,\qquad
 (1-\lambda)+\lambda U(s,a_i)<0.
\]
Thus $h(s)=0$ and the same postponement comparison forces $c$ at every future date.  Since $\lambda>\lambda_c$ implies $1-1/[\lambda(1+M)]>1/n$, the set $\mathcal T_\lambda$ contains $\bar\alpha$ and is relatively open and nonempty.  At $q_\lambda=0$, equality permits only compromise pairs and the pairs $(\pure{i},a_i)$.

For existence, define destinations $p_c=\bar\alpha$, $p_i=\pure{i}$ and scores
\begin{equation}\label{eq:nworldscores}
 g_c(\alpha)=0,\qquad
 g_i(\alpha)=1-\lambda(1+M)(1-\alpha^i),\qquad
 G(\alpha)=\max\{0,g_1(\alpha),\ldots,g_n(\alpha)\}.
\end{equation}
Let $r(\alpha)$ maximize these scores and put $\phi(\alpha)=p_{r(\alpha)}$, with any static-optimal action selection.  Each $p_i$ is absorbing because its own score is $1$ and all other scores are strictly lower.  Moreover,
\[
 \max_i g_i(\alpha)\geq\frac1n\sum_{i=1}^n g_i(\alpha)
 =q_\lambda.
\]
If $q_\lambda>0$, the destination $p_c$ is never selected.  If $q_\lambda<0$, it is absorbing, since its other scores all equal $q_\lambda$.  If $q_\lambda=0$, the common destination can be selected only at $\bar\alpha$: indeed $g_i(\alpha)=\lambda(1+M)(\alpha^i-1/n)$, so $\max_i g_i(\alpha)\leq0$ forces all weights to equal $1/n$.  At that state choose either $p_c$ or a pure destination.  Thus every destination actually selected is absorbing, in all three cases.

At every date of any deviation from a current state $\alpha$, an extreme action $a_i$ gives an assessment at most $g_i(\alpha)$, and $c$ gives zero.  Hence the deviation payoff is at most $G(\alpha)/(1-\delta)$.  Selecting $p_{r(\alpha)}$ attains this bound in every future period.  This verifies the equilibrium inequalities at every state and provides both continuations at the critical value.
\end{proof}

For three perspectives and $M=2$, the static region in which $c$ is optimal is $C_c=\{\alpha:\max_i\alpha^i\leq2/3\}$. When $\lambda>1/2$, every equilibrium starting from $\mathcal T_\lambda$ remains in $C_c$ and at $\ell^1$ distance at least $2/3$ from purity. The set $C_c$ constrains subsequent states; $\mathcal T_\lambda$ is the sufficient initial region, which contains the equal mixture.

\begin{proposition}[Asymmetric peak utilities]\label{prop:asymmetric}
The persistence of mixed worldviews does not rely on the two pure worldviews having the same maximal utility.  Consider instead
\begin{equation}\label{eq:asymmatrix}
\begin{array}{c|ccc}
 & c & a & b\\ \hline
 u_1 & 0 & 1 & -5\\
 u_2 & 0 & -3 & 2
\end{array}.
\end{equation}
Then pure worldview $2$ has strictly greater maximal utility than pure worldview $1$:
\[
 \max_{x\in X}u_2(x)=2>1=\max_{x\in X}u_1(x).
\]
Fix $\delta\in(0,1)$ and $\lambda\in(4/7,1)$.  From the initial worldview $\alpha_0=1/2$, every stationary MPE satisfies
\begin{equation}\label{eq:asymtrapconclusion}
 x_t=c,
 \qquad
 \alpha_t\in\left[\frac27,\frac34\right]
 \qquad\text{for every }t\geq0.
\end{equation}
In particular, no stationary MPE approaches the set of pure worldviews.  Moreover, a stationary pure-strategy MPE exists for every such $(\delta,\lambda)$.
\end{proposition}

\begin{proof}
For the environment \eqref{eq:asymmatrix},
\begin{equation}\label{eq:asymlines}
 U(\alpha,c)=0,
 \qquad
 U(\alpha,a)=4\alpha-3,
 \qquad
 U(\alpha,b)=2-7\alpha.
\end{equation}
Hence
\begin{equation}\label{eq:asymBcell}
 c\in\B(\alpha)
 \quad\Longleftrightarrow\quad
 \alpha\in\left[\frac27,\frac34\right],
\end{equation}
and $c$ is uniquely optimal at $\bar\alpha=1/2$.  From the perspective of $\bar\alpha$,
\[
 U\left(\frac12,a\right)=-1,
 \qquad
 U\left(\frac12,b\right)=-\frac32.
\]
For any future worldview $\beta$ and any $y\in\B(\beta)$, the corresponding one-period assessment is bounded as follows:
\begin{align}
 y=c
 &\implies
 (1-\lambda)U(\beta,c)+\lambda U\left(\frac12,c\right)=0,
 \label{eq:asymqc}\\
 y=a
 &\implies
 (1-\lambda)U(\beta,a)+\lambda U\left(\frac12,a\right)
 \leq 1-2\lambda<0,
 \label{eq:asymqa}\\
 y=b
 &\implies
 (1-\lambda)U(\beta,b)+\lambda U\left(\frac12,b\right)
 \leq 2-\frac72\lambda<0.
 \label{eq:asymqb}
\end{align}
The last strict inequality is exactly $\lambda>4/7$.  Let $C$ denote the equilibrium continuation payoff from $\bar\alpha$, with the leading discount factor removed as in \eqref{eq:Cmid}.  Equations \eqref{eq:asymqc}--\eqref{eq:asymqb} imply $C\leq0$.  Repeating $\bar\alpha$ for one period gives deviation payoff $\delta C$, so equilibrium optimality gives $C\geq\delta C$ and therefore $C\geq0$.  Thus $C=0$.  Since every noncompromise flow is strictly negative, every future action is $c$.  Equation \eqref{eq:asymBcell} proves \eqref{eq:asymtrapconclusion}.

To establish existence, we specify the policy at every worldview.  Define
\[
 p_c=\frac12,
 \qquad
 p_a=1,
 \qquad
 p_b=0,
\]
and
\begin{align}
 g_c(\alpha)&=0,\label{eq:asymgc}\\
 g_a(\alpha)&=(1-\lambda)+\lambda(4\alpha-3),\label{eq:asymga}\\
 g_b(\alpha)&=2(1-\lambda)+\lambda(2-7\alpha)
 =2-7\lambda\alpha.\label{eq:asymgb}
\end{align}
Let
\[
 G(\alpha)=\max\{g_c(\alpha),g_a(\alpha),g_b(\alpha)\},
\]
choose $r(\alpha)\in\argmax_{r\in\{c,a,b\}}g_r(\alpha)$, and set
\[
 \phi(\alpha)=p_{r(\alpha)},
 \qquad
 z(\alpha)\in\B(\alpha).
\]
The three targets are absorbing under this policy.  Indeed,
\begin{align*}
 &g_c(1/2)=0>
 \max\left\{1-2\lambda,2-\frac72\lambda\right\},\\
 &g_a(1)=1>\max\{0,2-7\lambda\},\\
 &g_b(0)=2>\max\{0,1-4\lambda\}.
\end{align*}
For an arbitrary current worldview $\alpha$, candidate next worldview $\beta$, and $y=z(\beta)$,
\[
 (1-\lambda)U(\beta,y)+\lambda U(\alpha,y)
 \leq g_y(\alpha)
 \leq G(\alpha),
\]
because $U(\beta,a)\leq1$, $U(\beta,b)\leq2$, and $U(\beta,c)=0$.  From the following period onward, the policy remains at one of the three absorbing targets, and every later one-period assessment is also at most $G(\alpha)$.  Consequently,
\[
 C_{\phi,z}(\beta;\alpha)
 \leq\frac{G(\alpha)}{1-\delta}.
\]
Choosing $\beta=p_{r(\alpha)}$ produces the constant flow $G(\alpha)$ and attains this upper bound.  Hence $(\phi,z)$ is a stationary MPE.
\end{proof}

\begin{proposition}[Perturbations preserving the common value]\label{prop:common-value-perturbations}
Fix $\lambda\in(4/7,1)$ and put
\[
 \eta=\min\{1,2\lambda-1,7\lambda/2-2\}>0,
 \qquad \varepsilon=\min\{1/8,\eta/16\}.
\]
Consider the payoff array and initial worldview
\[
\begin{array}{c|ccc}
 &c&a&b\\ \hline
 u_1&0&A_1&B_1\\
 u_2&0&A_2&B_2
\end{array},\qquad \alpha_0=s,
\]
where
\begin{equation}\label{eq:perturbball}
 \max\{|A_1-1|,|A_2+3|,|B_1+5|,|B_2-2|,|s-1/2|\}
 <\varepsilon.
\end{equation}
For every $\delta\in(0,1)$ a stationary MPE exists.  In every stationary MPE starting from $s$,
\begin{equation}\label{eq:perturbedpersistence}
 x_t=c,\qquad \ell\leq\alpha_t\leq r\quad(t\geq0),
 \qquad
 \ell=\frac{B_2}{B_2-B_1},\quad
 r=\frac{-A_2}{A_1-A_2},
\end{equation}
with $0<\ell<s<r<1$.  The distance from each future worldview to $\{0,1\}$ is at least $\min\{\ell,1-r\}>0$.  The pure-worldview maximal utilities remain strictly ordered, with $B_2>A_1$.
\end{proposition}
\begin{proof}
The perturbation bound gives $s\in(0,1)$, $A_1,B_2>0$, and $A_2,B_1<0$.  Therefore
\[
 \max_x u_1(x)=A_1,\qquad \max_x u_2(x)=B_2,
 \qquad B_2-A_1>1-2\varepsilon\geq\frac34>0.
\]
Define
\begin{align*}
 g_a(\alpha)&=(1-\lambda)A_1+\lambda[\alpha A_1+(1-\alpha)A_2],\\
 g_b(\alpha)&=(1-\lambda)B_2+\lambda[\alpha B_1+(1-\alpha)B_2],
 \qquad g_c(\alpha)=0.
\end{align*}
For the unperturbed values, the four gaps
\[
 -g_a(s),\quad -g_b(s),\quad A_1-g_b(1),\quad B_2-g_a(0)
\]
are respectively $2\lambda-1$, $7\lambda/2-2$, $7\lambda-1$, and $1+4\lambda$, each at least $\eta$.  Under \eqref{eq:perturbball}, the changes in these four gaps have absolute values at most $5\varepsilon$, $8\varepsilon$, $2\varepsilon$, and $2\varepsilon$.  To see the bounds at $s$, write, for example,
\begin{align*}
 |sA_1+(1-s)A_2-(-1)|
 &\leq s|A_1-1|+(1-s)|A_2+3|+4|s-1/2|
 <5\varepsilon,\\
 |sB_1+(1-s)B_2-(-3/2)|
 &\leq s|B_1+5|+(1-s)|B_2-2|+7|s-1/2|
 <8\varepsilon.
\end{align*}
Consequently,
\begin{align*}
 |g_a(s)-(1-2\lambda)|
 &<(1-\lambda)\varepsilon+5\lambda\varepsilon\leq5\varepsilon,\\
 |g_b(s)-(2-7\lambda/2)|
 &<(1-\lambda)\varepsilon+8\lambda\varepsilon\leq8\varepsilon.
\end{align*}
For the remaining two gaps,
\begin{align*}
 |[A_1-g_b(1)]-(7\lambda-1)|
 &\leq |A_1-1|+(1-\lambda)|B_2-2|+\lambda|B_1+5|
 <2\varepsilon,\\
 |[B_2-g_a(0)]-(1+4\lambda)|
 &\leq |B_2-2|+(1-\lambda)|A_1-1|+\lambda|A_2+3|
 <2\varepsilon.
\end{align*}  Since $8\varepsilon\leq\eta/2$, all four gaps remain strictly positive:
\begin{equation}\label{eq:perturbstablegaps}
 g_a(s)<0,\qquad g_b(s)<0,\qquad
 A_1>g_b(1),\qquad B_2>g_a(0).
\end{equation}
Also $A_1>A_2$ and $B_2>B_1$.  Thus the assessment at $s$ of any future action $a$ or $b$ is at most $g_a(s)$ or $g_b(s)$, respectively, whereas the assessment of $c$ is zero.  The inequalities $g_a(s),g_b(s)<0$, together with $A_1,B_2>0$, imply $U(s,a),U(s,b)<0$ and hence $h(s)=0$.  In an arbitrary stationary MPE, let $C=C_{\phi,z}(\phi(s);s)$.  Then
\[
 C\leq0,\qquad C\geq h(s)+\delta C=\delta C,
 \qquad\text{so }C=0.
\]
Every future action must therefore be $c$.  Since $A_1-A_2>0$ and $B_2-B_1>0$, static optimality gives
\begin{equation}\label{eq:perturbcell}
 \begin{split}
 c\in\B(\alpha)
 &\iff \alpha(A_1-A_2)+A_2\leq0
       \ \text{and}\ B_2-\alpha(B_2-B_1)\leq0\\
 &\iff \ell\leq\alpha\leq r.
 \end{split}
\end{equation}
The signs of $A_1,A_2,B_1,B_2$ imply $0<\ell,r<1$.  The strict inequalities $U(s,a),U(s,b)<0$ give $\ell<s<r$.  Consequently $\dist(\alpha_t,\{0,1\})\geq\min\{\ell,1-r\}>0$ for all $t\geq0$.

For existence, take destinations $p_c=s$, $p_a=1$, $p_b=0$, and let $\phi(\alpha)$ select $p_y$ for any $y$ maximizing $g_y(\alpha)$.  Choose $z(\alpha)\in z^*(\alpha)$.  Each destination is absorbing by \eqref{eq:perturbstablegaps}, since $g_a(1)=A_1>0$ and $g_b(0)=B_2>0$.  The current action at each destination is respectively $c,a,b$.  For every successor $\beta$ and every date along its continuation,
\[
 (1-\lambda)U(\phi^k(\beta),z(\phi^k(\beta)))
 +\lambda U(\alpha,z(\phi^k(\beta)))
 \leq G(\alpha):=\max\{0,g_a(\alpha),g_b(\alpha)\}.
\]
Thus $C_{\phi,z}(\beta;\alpha)\leq G(\alpha)/(1-\delta)$, with equality when $\beta=\phi(\alpha)$ is the selected absorbing destination.  This verifies the MPE conditions at every state.
\end{proof}

\begin{remark}[The additional action]\label{rem:two-action-benchmark}
Each pure worldview in \eqref{eq:countermatrix} has a distinct, uniquely optimal action.  The additional action $c$ is optimal at mixed worldviews but not at either pure worldview.  This is the menu feature that distinguishes the example from the two-action cases analyzed in \citet[Section~II.A, pp.~727--731]{BBMZ2021}.
\end{remark}

\section{A sufficient condition for persistent mixed worldviews}\label{sec:persistence}

The upper bound \eqref{eq:flowceiling} applies to every statically implementable future outcome. Repeating the current worldview gives the reverse bound on the equilibrium continuation value.

\begin{proof}[Proof of \cref{thm:generaltrap}]
Let
\[
 \alpha_1=\phi(\bar\alpha),
 \qquad
 \alpha_{t+1}=\phi(\alpha_t),
 \qquad
 x_t=z(\alpha_t).
\]
Set
\begin{equation}\label{eq:Cgeneraltrap}
 C=C_{\phi,z}(\alpha_1;\bar\alpha)
 =\sum_{t=1}^{\infty}\delta^{t-1}
 q_{\bar\alpha}(\alpha_t,x_t).
\end{equation}
The upper-bound condition gives
\begin{equation}\label{eq:Cgeneralupper}
 C\leq\frac{h(\bar\alpha)}{1-\delta}.
\end{equation}
The current self can choose $\bar\alpha$ once more.  The next-period action is $z(\bar\alpha)\in\B(\bar\alpha)$, so its assessment is
\[
 (1-\lambda)h(\bar\alpha)+\lambda h(\bar\alpha)
 =h(\bar\alpha).
\]
At the following date the stationary policy chooses $\phi(\bar\alpha)=\alpha_1$, after which the original continuation is reproduced.  The deviation payoff is
\[
 h(\bar\alpha)+\delta C.
\]
Equilibrium optimality implies
\[
 C\geq h(\bar\alpha)+\delta C,
\]
and hence
\begin{equation}\label{eq:Cgenerallower}
 C\geq\frac{h(\bar\alpha)}{1-\delta}.
\end{equation}
Combining \eqref{eq:Cgeneralupper} and \eqref{eq:Cgenerallower},
\[
 0=\frac{h(\bar\alpha)}{1-\delta}-C
   =\sum_{t=1}^{\infty}\delta^{t-1}
        \bigl[h(\bar\alpha)-q_{\bar\alpha}(\alpha_t,x_t)\bigr].
\]
Every bracket is nonnegative and every coefficient is positive.  Thus each bracket is zero, which proves \eqref{eq:trapconclusionpair} and \eqref{eq:trapconclusionstate}.

For each $x\in X$, define
\[
 C_x=\{\beta\in\DeltaJ: U(\beta,x)\geq U(\beta,y)
                         \text{ for every }y\in X\}.
\]
Each $C_x$ is closed, being a finite intersection of closed half-spaces relative to $\DeltaJ$.  Hence
\[
 K(\bar\alpha)
 =\bigcup_{x\in X}\bigl(C_x\cap
       \{\beta:q_{\bar\alpha}(\beta,x)=h(\bar\alpha)\}\bigr)
\]
is closed in the compact simplex and therefore compact.  It is nonempty: any $x\in\B(\bar\alpha)$ gives $q_{\bar\alpha}(\bar\alpha,x)=h(\bar\alpha)$, so $\bar\alpha\in K(\bar\alpha)$.  If it is disjoint from the finite pure-state set, then
\[
 \varepsilon:=\min_{\beta\in K(\bar\alpha)}\min_{j\in J}
       \|\beta-\pure{j}\|_1
       =\dist\bigl(K(\bar\alpha),\pures\bigr)>0.
\]
Equation \eqref{eq:trapconclusionstate} gives
\[
 \dist\bigl(\alpha_t,\pures\bigr)\geq\varepsilon
\]
for every $t\geq1$.
\end{proof}

A sufficient condition for persistence with a common-value action can be checked by finitely many strict comparisons.

\begin{corollary}[Common-value compromise]\label{cor:commoncompromise}
Suppose there is an action $c\in X$ and a number $\bar c\in\R$ such that
\begin{equation}\label{eq:commonvalue}
 u_j(c)=\bar c
 \qquad\text{for every }j\in J.
\end{equation}
Let $\bar\alpha\in\DeltaJ$ satisfy
\begin{equation}\label{eq:coptimalbar}
 c\in\B(\bar\alpha),
 \qquad
 h(\bar\alpha)=\bar c.
\end{equation}
For $x\in X$, write
\[
 M(x)=\max_{j\in J}u_j(x).
\]
Assume that, for every implementable $x\neq c$,
\begin{equation}\label{eq:strictcompromiseineq}
 (1-\lambda)M(x)+\lambda U(\bar\alpha,x)
 <\bar c.
\end{equation}
Then, in every stationary MPE starting at $\bar\alpha$,
\begin{equation}\label{eq:compromiseallfuture}
 x_t=c
 \quad\text{and}\quad
 \alpha_t\in\{\alpha:c\in\B(\alpha)\}
 \qquad\text{for every }t\geq0.
\end{equation}
If $c$ is not optimal under any pure worldview, no stationary MPE starting at $\bar\alpha$ approaches the set of pure worldviews.
\end{corollary}

\begin{proof}
If $x=c$, \eqref{eq:commonvalue} gives
\[
 q_{\bar\alpha}(\beta,c)=\bar c
\]
for every $\beta$ at which $c$ is optimal.  If $x\neq c$, then
\[
 U(\beta,x)\leq M(x),
\]
so \eqref{eq:strictcompromiseineq} yields
\[
 q_{\bar\alpha}(\beta,x)
 \leq (1-\lambda)M(x)+\lambda U(\bar\alpha,x)
 <\bar c.
\]
The upper-bound condition holds, and equality is possible only for action $c$.  By \cref{thm:generaltrap}, every future action is $c$.  At $t=0$ the choice is also uniquely $c$: if $x\neq c$ were optimal at $\bar\alpha$, then
\[
 (1-\lambda)M(x)+\lambda U(\bar\alpha,x)
 \geq U(\bar\alpha,x)=\bar c,
\]
contrary to \eqref{eq:strictcompromiseineq}.  The equality set is exactly the closed action region $C_c=\{\alpha:c\in\B(\alpha)\}$.  If it contains no pure worldview, the last assertion follows from \cref{thm:generaltrap}.
\end{proof}

\begin{remark}[Threshold form]\label{rem:thresholdcompromise}
Suppose, for every $x\in\Ximp\setminus\{c\}$,
\[
 U(\bar\alpha,x)<\bar c<M(x).
\]
For each such $x$, rearrangement gives
\begin{align*}
 (1-\lambda)M(x)+\lambda U(\bar\alpha,x)<\bar c
 &\iff M(x)-\bar c<\lambda[M(x)-U(\bar\alpha,x)].
\end{align*}
The denominator is positive, so \eqref{eq:strictcompromiseineq} is equivalent to
\begin{equation}\label{eq:lambdacompthreshold}
 \lambda>
 \frac{M(x)-\bar c}{M(x)-U(\bar\alpha,x)}.
\end{equation}
The right-hand side lies strictly between zero and one.  For sufficiently high $\lambda$, every stationary MPE from $\bar\alpha$ therefore sustains the compromise action.  In \eqref{eq:countermatrix}, $M(a)=M(b)=1$, $U(1/2,a)=U(1/2,b)=-1$, and $\bar c=0$, so \eqref{eq:lambdacompthreshold} becomes $\lambda>1/2$.
\end{remark}

\subsection{A finite characterization of the upper-bound condition}\label{sec:ceilinggeometry}

The condition in Theorem~\ref{thm:generaltrap} quantifies over a continuum of future worldviews, but it can be checked without computing an equilibrium.  For each $x\in X^I$, let
\begin{equation}\label{eq:implementedmaximum}
 C_x=\{\beta\in\Delta^n:x\in z^*(\beta)\},\qquad
 m_x=\max_{\beta\in C_x}U(\beta,x),\qquad
 F_x=\argmax_{\beta\in C_x}U(\beta,x).
\end{equation}
Each $C_x$ is a nonempty compact polytope.  Thus $m_x$ is the value of a linear program and $F_x$ is a nonempty exposed face of $C_x$, possibly the whole polytope.

\begin{proposition}[Linear-program characterization]\label{prop:ceilingLP}
For a given $\bar\alpha$, condition~\eqref{eq:flowceiling} holds if and only if
\begin{equation}\label{eq:finiteceiling}
 (1-\lambda)m_x+\lambda U(\bar\alpha,x)\leq h(\bar\alpha)
 \qquad\text{for every }x\in X^I.
\end{equation}
When it holds, define
\[
 I(\bar\alpha)=\{x\in X^I:(1-\lambda)m_x+\lambda U(\bar\alpha,x)=h(\bar\alpha)\}.
\]
Then
\begin{equation}\label{eq:Kfaces}
 K(\bar\alpha)=\bigcup_{x\in I(\bar\alpha)}F_x.
\end{equation}
In particular, both the upper-bound condition and exclusion of all pure worldviews from its equality set have finite descriptions in the primitive payoffs.
\end{proposition}
\begin{proof}
For fixed $x$, maximizing $q_{\bar\alpha}(\beta,x)$ over $\beta\in C_x$ gives $(1-\lambda)m_x+\lambda U(\bar\alpha,x)$, since $1-\lambda>0$.  Taking the maximum over the finite set $X^I$ proves the equivalence.  If $x\notin I(\bar\alpha)$, its assessment is strictly below $h(\bar\alpha)$ for every $\beta\in C_x$.  If $x\in I(\bar\alpha)$, then
\[
 h(\bar\alpha)-q_{\bar\alpha}(\beta,x)
 =(1-\lambda)[m_x-U(\beta,x)],
\]
which is zero exactly on $F_x$.  This proves~\eqref{eq:Kfaces}.  The programs defining $m_x$ have constraints
\[
 \beta^j\geq0,\qquad\sum_j\beta^j=1,\qquad
 \sum_j\beta^j[u_j(x)-u_j(y)]\geq0\quad(y\in X).
\]
A pure worldview belongs to $F_x$ exactly when it belongs to $C_x$ and gives payoff $m_x$ to $x$.  These are finitely many comparisons, one for each pair $(j,x)$.
\end{proof}

The distinction between $m_x$ and $M(x)=\max_j u_j(x)$ matters.  The latter may be attained only at a worldview that would choose another action.  The linear program restricts the maximization to worldviews that can actually implement $x$, so~\eqref{eq:finiteceiling} is exact, whereas substituting $M(x)$ gives a sufficient bound.

\begin{proposition}[Payoff equalities required by the upper bound]\label{prop:ceilingequalities}
Suppose $0<\lambda<1$ and $\bar\alpha$ satisfies~\eqref{eq:flowceiling}.  For every $x\in z^*(\bar\alpha)$ and any $i,j$ with $\bar\alpha^i>0$ and $\bar\alpha^j>0$,
\begin{equation}\label{eq:facepayoffequality}
 u_i(x)=u_j(x).
\end{equation}
Consequently, if $\bar\alpha$ has full support, every action optimal there has a worldview-independent payoff.  If the values $u_1(x),\ldots,u_n(x)$ are pairwise distinct for every action $x$, no mixed worldview satisfies~\eqref{eq:flowceiling}.
\end{proposition}
\begin{proof}
Write $A=z^*(\bar\alpha)$ and take a vector $d$ supported on $\{i:\bar\alpha^i>0\}$ with $\sum_i d^i=0$.  Put $s_x=\sum_i d^i u_i(x)$ for every $x\in X$, and $R=\max_{x\in X}|s_x|$.  Choose $\varepsilon>0$ so that $\varepsilon|d^i|<\bar\alpha^i$ for every nonzero coordinate of $d$.  This makes both $\bar\alpha+\varepsilon d$ and $\bar\alpha-\varepsilon d$ feasible.  If $X\setminus A$ is nonempty, put
\[
 \Delta=\min_{y\notin A}[h(\bar\alpha)-U(\bar\alpha,y)]>0
\]
and additionally choose $2\varepsilon R<\Delta$.  For $x\in A$, $y\notin A$, and either sign,
\[
 U(\bar\alpha\pm\varepsilon d,x)-U(\bar\alpha\pm\varepsilon d,y)
 \geq\Delta-2\varepsilon R>0.
\]
Hence every maximizing action at either perturbed state belongs to $A$.  If $A=X$ this conclusion requires no gap bound.
At $\bar\alpha+\varepsilon d$, an optimal action $y\in A$ maximizes $s_x$ over $A$.  Its assessment from $\bar\alpha$ is
\[
 q_{\bar\alpha}(\bar\alpha+\varepsilon d,y)
 =h(\bar\alpha)+(1-\lambda)\varepsilon\max_{x\in A}s_x.
\]
The upper-bound condition implies $\max_{x\in A}s_x\leq0$.  Applying the same argument to $-d$ gives $\min_{x\in A}s_x\geq0$.  Therefore $s_x=0$ for every $x\in A$.  For distinct $i,j$ with positive weights, take $d=\pure{i}-\pure{j}$. This is a direction in the linear space parallel to the affine hull of the simplex:
\[
 \sum_k d^k=0,\qquad
 \sum_k d^k u_k(x)=u_i(x)-u_j(x).
\]
The equality $s_x=0$ therefore proves~\eqref{eq:facepayoffequality}. A mixed worldview has at least two positive coordinates, which proves the last assertion.
\end{proof}

These equalities follow from the pointwise upper bound. At equilibrium accumulation points, a stronger restriction holds without that bound: supported worldviews must attain the actionwise valuation maximum (\cref{sec:contactgeometry}). The distinction between exact and approximate agreement is developed in \cref{cor:nearcommon}.

\section{Convergence of experienced utility and agreement over future actions}\label{sec:alignment}

The same postponement option restricts equilibrium behavior even when the persistence condition fails. Throughout this section, $h_t,H_t,Q_t,A_t$ refer to the quantities in \eqref{eq:ht}--\eqref{eq:At} along a fixed stationary MPE path.

\begin{proof}[Proof of \cref{thm:alignment}]
Fix $t$.  The date-$t$ self can deviate by choosing the current worldview $\alpha_t$ as the next worldview.  Under this deviation, the date-$(t+1)$ state is again $\alpha_t$, so the date-$(t+1)$ action is $x_t=z(\alpha_t)$ and the first future one-period assessment is $h_t$.  At date $t+1$, stationarity implies that the next self chooses $\phi(\alpha_t)=\alpha_{t+1}$, after which the original equilibrium continuation is reproduced with one-period delay.  Therefore the normalized payoff from the delay deviation is
\begin{equation}\label{eq:delaypayoffalignment}
 (1-\delta)h_t+\delta A_t.
\end{equation}
Equilibrium optimality gives
\[
 A_t\geq(1-\delta)h_t+\delta A_t.
\]
Since $1-\delta>0$, this is equivalent to \eqref{eq:Atgeht}.

Because $x_t$ is statically optimal under $\alpha_t$,
\begin{equation}\label{eq:Qtupper}
 U(\alpha_t,x_{t+k})\leq h_t
 \qquad\text{for every }k\geq1.
\end{equation}
Averaging \eqref{eq:Qtupper} gives
\begin{equation}\label{eq:Qtupperht}
 Q_t\leq h_t.
\end{equation}
Combining \eqref{eq:At}, \eqref{eq:Atgeht}, and \eqref{eq:Qtupperht},
\[
 h_t
 \leq(1-\lambda)H_{t+1}+\lambda Q_t
 \leq(1-\lambda)H_{t+1}+\lambda h_t.
\]
Since $1-\lambda>0$,
\begin{equation}\label{eq:Hnextgeht}
 H_{t+1}\geq h_t.
\end{equation}
From \eqref{eq:Ht},
\begin{equation}\label{eq:Hrecursion}
 H_t=(1-\delta)h_t+\delta H_{t+1}.
\end{equation}
Hence
\begin{equation}\label{eq:Hdiff}
 H_{t+1}-H_t
 =(1-\delta)(H_{t+1}-h_t)
 \geq0,
\end{equation}
which proves \eqref{eq:Hmonotone}.

Let $\underline u=\min_{j,x}u_j(x)$ and $\overline u=\max_{j,x}u_j(x)$.  Convexity of the averaging weights gives
\[
 \underline u\leq h_t\leq\overline u,
 \qquad \underline u\leq H_t\leq\overline u.
\]
Thus $(H_t)$ is bounded.  Therefore
\[
 H_t\longrightarrow L
\]
for some $L$.  Equation \eqref{eq:Hdiff} implies
\[
 H_{t+1}-h_t
 =\frac{H_{t+1}-H_t}{1-\delta}
 \longrightarrow0,
\]
so $h_t\to L$.  The bounds
\[
 h_t\leq A_t
 \leq(1-\lambda)H_{t+1}+\lambda h_t
\]
then imply $A_t\to L$.

Using \eqref{eq:At},
\begin{align}
 \lambda(h_t-Q_t)
 &=\lambda h_t-\lambda Q_t\notag\\
 &=(1-\lambda)H_{t+1}+\lambda h_t-A_t\notag\\
 &\leq(1-\lambda)(H_{t+1}-h_t),
 \label{eq:Qgapbound}
\end{align}
where the inequality uses $A_t\geq h_t$.  The left-hand side is nonnegative by \eqref{eq:Qtupperht}, and the right-hand side converges to zero.  Since $\lambda>0$,
\begin{equation}\label{eq:Qconverges}
 h_t-Q_t\longrightarrow0,
\end{equation}
which gives $Q_t\to L$.

Finally, by \eqref{eq:Qt},
\begin{equation}\label{eq:gapsum}
 h_t-Q_t
 =(1-\delta)\sum_{r=1}^{\infty}\delta^{r-1}
 \left[h_t-U(\alpha_t,x_{t+r})\right].
\end{equation}
Every summand in brackets is nonnegative by static optimality.  For fixed $k\geq1$,
\begin{equation}\label{eq:fixedgapbound}
 0\leq
 (1-\delta)\delta^{k-1}
 \left[h_t-U(\alpha_t,x_{t+k})\right]
 \leq h_t-Q_t.
\end{equation}
Letting $t\to\infty$ and using \eqref{eq:Qconverges} proves \eqref{eq:alignmentlimit}.

For the stronger summability assertion, combine \eqref{eq:fixedgapbound}, \eqref{eq:Qgapbound}, and \eqref{eq:Hdiff}:
\begin{equation}\label{eq:lossvariation}
 0\leq h_t-U(\alpha_t,x_{t+k})
 \leq \frac{1-\lambda}{\lambda(1-\delta)^2\delta^{k-1}}
                 (H_{t+1}-H_t).
\end{equation}
Let $c_k=(1-\lambda)/[\lambda(1-\delta)^2\delta^{k-1}]$.  For every $T\geq0$,
\[
 \sum_{t=0}^{T}\bigl[h_t-U(\alpha_t,x_{t+k})\bigr]
 \leq c_k\sum_{t=0}^{T}(H_{t+1}-H_t)
 =c_k(H_{T+1}-H_0).
\]
The terms on the left are nonnegative.  Taking $T\to\infty$, and using $H_{T+1}\to L$ and $0\leq L-H_0\leq\overline u-\underline u$, proves both inequalities in \eqref{eq:totalgap}.
\end{proof}

\begin{corollary}[Cluster-point compatibility]\label{cor:clustercompat}
Let $\alpha_{t_r}\to\alpha^*$ along a subsequence.  Fix $k\geq1$ and pass, if necessary, to a further subsequence on which
\[
 x_{t_r+k}=x
\]
is constant.  Then
\begin{equation}\label{eq:clustercompat}
 x\in\B(\alpha^*),
 \qquad
 U(\alpha^*,x)=h(\alpha^*)=L.
\end{equation}
\end{corollary}

\begin{proof}
By \cref{thm:alignment},
\[
 h_{t_r}\to L
 \quad\text{and}\quad
 h_{t_r}-U(\alpha_{t_r},x)\to0.
\]
Continuity gives
\[
 U(\alpha^*,x)=L.
\]
Since $h$ is continuous,
\[
 h(\alpha^*)=\lim_{r\to\infty}h(\alpha_{t_r})=L.
\]
Thus $x$ attains the maximum at $\alpha^*$.
\end{proof}

\begin{corollary}[Optimal actions on periodic paths]\label{cor:periodic}
Suppose a stationary MPE path is periodic with an integer period $p\geq1$.  Let
\[
 \mathcal{A}=\{x_0,\dots,x_{p-1}\}
\]
be the set of actions used on the cycle.  Then
\begin{equation}\label{eq:allactionsoptimalcycle}
 \mathcal{A}\subseteq\B(\alpha_t)
 \qquad\text{for every }t.
\end{equation}
In particular, if $\B(\alpha_t)$ is a singleton at every state on the cycle, then the action is constant along the cycle.
\end{corollary}

\begin{proof}
Fix a residue $s\in\{0,\ldots,p-1\}$.  Periodicity and \cref{thm:alignment} give
\[
 h_s=h_{s+mp}\longrightarrow L,
 \qquad\text{hence }h_s=L.
\]
For each $k\in\{1,\ldots,p\}$,
\[
 h_s-U(\alpha_s,x_{s+k})
 =h_{s+mp}-U(\alpha_{s+mp},x_{s+mp+k})
 \longrightarrow0.
\]
The left-hand side is independent of $m$ and therefore equals zero.  The indices $s+1,\ldots,s+p$ cover every action used in a period.  Hence every action in $\mathcal A$ attains $h(\alpha_s)$, for every residue $s$, proving \eqref{eq:allactionsoptimalcycle}.
\end{proof}

\begin{remark}[Actions and worldviews]
A periodic path with more than one action lies entirely in the intersection of those actions' static optimality regions.  If the action is constant on a periodic path, every visited worldview assigns it the same utility $L$; the conclusion does not require the worldviews themselves to coincide.  Static indifference at a mixed worldview is compatible with distinct pure-worldview payoffs, so the restriction in \cref{cor:periodic} differs from \eqref{eq:genericity}.
\end{remark}

\begin{corollary}[Dates with a given evaluation loss]\label{cor:gapbound}
For every integer $k\geq1$ and every $\eta>0$,
\begin{equation}\label{eq:lossdatecount}
 \#\{t\geq0:h_t-U(\alpha_t,x_{t+k})\geq\eta\}
 \leq\frac{(1-\lambda)(L-H_0)}{\eta\lambda(1-\delta)^2\delta^{k-1}}.
\end{equation}
In particular, only finitely many dates have a loss of at least $\eta$.
\end{corollary}
\begin{proof}
Let $N_T$ count the dates $0\leq t\leq T$ at which the loss is at least $\eta$.  Nonnegativity of the losses and \eqref{eq:totalgap} give
\[
 \eta N_T\leq\sum_{t=0}^{T}\bigl[h_t-U(\alpha_t,x_{t+k})\bigr]
 \leq\frac{(1-\lambda)(L-H_0)}{\lambda(1-\delta)^2\delta^{k-1}}.
\]
The sequence $N_T$ is increasing.  Taking its limit proves \eqref{eq:lossdatecount}.
\end{proof}

The monotone quantity is the discounted average $H_t$; the realized sequence $h_t$ need only converge.  For any fixed horizon and positive loss threshold, \cref{cor:gapbound} bounds the number of dates at which the future action is worse than the current optimum by at least that threshold.  The bound compares the consumer's successive evaluations along a single equilibrium path.
\section{The continuation value and limiting worldviews}\label{sec:contactgeometry}

The preceding section establishes convergence of experienced utility without imposing a payoff restriction.  To learn where the worldview itself can accumulate, we use an additional implication of equilibrium: the continuation value is the upper envelope of affine functions of the current worldview.  Its regularity does not require continuity of the equilibrium policy.

\subsection{Convexity and a supporting continuation}
For the fixed stationary MPE $(\phi,z)$, write $S(\alpha)=C_{\phi,z}(\phi(\alpha);\alpha)$ as in \eqref{eq:contactdefinition}.  Let $B=\max_{j,x}|u_j(x)|$.

\begin{lemma}[The continuation-value envelope]\label{lem:convexvalue}
The function $S$ is convex and satisfies
\begin{equation}\label{eq:Slipschitz}
 |S(\alpha)-S(\eta)|\leq\frac{\lambda B}{1-\delta}\|\alpha-\eta\|_1.
\end{equation}
For each $\beta$, the affine function
\begin{equation}\label{eq:supportingcontinuation}
 f_\beta(\alpha)=C_{\phi,z}(\phi(\beta);\alpha)
              =S(\beta)+b_\beta\mathbin{\cdot}(\alpha-\beta)
\end{equation}
supports $S$ from below, where
\begin{equation}\label{eq:supportslope}
 b_\beta^j=\lambda\sum_{k=0}^{\infty}\delta^k
              u_j\bigl(z(\phi^{k+1}(\beta))\bigr),\qquad
 \|b_\beta\|_\infty\leq\frac{\lambda B}{1-\delta}.
\end{equation}
In particular, $S(\alpha)\geq f_\beta(\alpha)$ for all $\alpha$, with equality at $\alpha=\beta$.  Moreover,
\begin{equation}\label{eq:Bellmansupportidentity}
 C_{\phi,z}(\beta;\alpha)
 =(1-\lambda)h(\beta)+\lambda U(\alpha,z(\beta))+\delta f_\beta(\alpha).
\end{equation}
\end{lemma}
\begin{proof}
For a fixed successor $\gamma$, expand the current-worldview component of the absolutely convergent series:
\[
 C_{\phi,z}(\gamma;\alpha)
 =(1-\lambda)\sum_{k=0}^{\infty}\delta^k h(\phi^k(\gamma))
 +\lambda\sum_{j\in J}\alpha^j
       \sum_{k=0}^{\infty}\delta^k u_j(z(\phi^k(\gamma))).
\]
This is affine in $\alpha$.  Equilibrium optimality gives
\begin{equation}\label{eq:Ssupremum}
 S(\alpha)=\sup_{\gamma\in\Delta^n}C_{\phi,z}(\gamma;\alpha),
\end{equation}
where the supremum is attained at $\phi(\alpha)$.  Taking the supremum of affine functions proves convexity.  For every $\gamma$,
\[
 |C_{\phi,z}(\gamma;\alpha)-C_{\phi,z}(\gamma;\eta)|
 \leq\lambda\sum_{k=0}^{\infty}\delta^k B\|\alpha-\eta\|_1
 =\frac{\lambda B}{1-\delta}\|\alpha-\eta\|_1.
\]
Taking the supremum over $\gamma$ yields
\[
 S(\alpha)\leq S(\eta)+\frac{\lambda B}{1-\delta}\|\alpha-\eta\|_1.
\]
Interchanging $\alpha$ and $\eta$ proves \eqref{eq:Slipschitz}. Choosing $\gamma=\phi(\beta)$ in \eqref{eq:Ssupremum} gives the supporting inequality.  Its slope is exactly \eqref{eq:supportslope}, and equality at $\beta$ follows from the definition of $S(\beta)$.  Finally, separating the $k=0$ term of $C_{\phi,z}(\beta;\alpha)$ and using $U(\beta,z(\beta))=h(\beta)$ proves \eqref{eq:Bellmansupportidentity}.
\end{proof}

The following estimate controls the error of a supporting continuation along a feasible direction, including at nondifferentiable points of $S$.

\begin{lemma}[Supporting error on a feasible ray]\label{lem:rayerror}
Let $S:\Delta^n\to\R$ be convex and Lipschitz. Suppose $p\in\Delta^n$, $d\in\R^n$, and $t_0>0$ satisfy $p+t d\in\Delta^n$ for $0\leq t\leq t_0$. Write $\beta_t=p+t d$. For each $0<t\leq t_0$, let $b_{\beta_t}\in\R^n$ satisfy
\[
 S(\eta)\geq S(\beta_t)+b_{\beta_t}\cdot(\eta-\beta_t)
 \qquad(\eta\in\Delta^n).
\]
Define
\begin{equation}\label{eq:rayresidual}
 R_t=S(p)-S(\beta_t)+t\,b_{\beta_t}\mathbin{\cdot}d\geq0.
\end{equation}
Then $R_t/t\longrightarrow0$ as $t\downarrow0$.
\end{lemma}
\begin{proof}
Set $g(t)=S(p+t d)$.  The function $g$ is convex and Lipschitz.  For $0<s<t$, convexity gives
\[
 g(s)\leq\left(1-\frac{s}{t}\right)g(0)+\frac{s}{t}g(t),
 \qquad
 \frac{g(s)-g(0)}{s}\leq\frac{g(t)-g(0)}{t}.
\]
If $L_S$ is a Lipschitz constant for $S$ in the $\ell^1$ norm, these quotients have absolute value at most $L_S\|d\|_1$. They therefore have a finite limit as $t\downarrow0$.  For $0<2t\leq t_0$, the supporting inequality at $\beta_t$, evaluated at $p$ and at $p+2t d$, yields
\[
 \frac{g(t)-g(0)}{t}
 \leq b_{\beta_t}\mathbin{\cdot}d
 \leq\frac{g(2t)-g(t)}{t}.
\]
Therefore
\begin{align*}
 0\leq\frac{R_t}{t}
 &\leq\frac{g(2t)-2g(t)+g(0)}{t}\\
 &=2\left\{\frac{g(2t)-g(0)}{2t}
                 -\frac{g(t)-g(0)}{t}\right\}\longrightarrow0.
\end{align*}
Both quotients on the last line converge to the same one-sided limit.
\end{proof}

\subsection{Contact with the current-payoff bound}
\Cref{lem:postponement} gives $(1-\delta)S\geq h$.  A point of equality places a restriction on all actions optimal there, including actions not selected by a particular tie-breaking rule.

\begin{lemma}[Actions at a contact state]\label{lem:contactmax}
If $p\in\mathcal Z$, then \eqref{eq:contactmaxima} holds for every $x\in z^*(p)$.
\end{lemma}
\begin{proof}
Write $h_0=h(p)$ and $A=z^*(p)$.  Fix $j\in J$ and take the feasible direction $d=\alpha(j)-p$.  Define
\[
 s_x=U(d,x)=u_j(x)-U(p,x),\qquad
 r=\max_{x\in A}s_x,\qquad
 R=\max_{x\in X}|s_x|,
\]
where $U(d,x)=\sum_i d^i u_i(x)$ denotes the linear extension to direction vectors.  If $X\setminus A\ne\varnothing$, let
\[
 \eta=\min_{y\notin A}\{h_0-U(p,y)\}>0.
\]
For $x\in A$, $y\notin A$, and $t>0$ sufficiently small that $2tR<\eta$,
\[
 U(p+t d,x)-U(p+t d,y)
 =h_0-U(p,y)+t(s_x-s_y)\geq\eta-2tR>0.
\]
If $R=0$ the inequality holds without a restriction involving $R$; if $A=X$ there are no inactive actions to exclude.  Consequently, for every sufficiently small $t>0$, every optimal action at $\beta_t=p+t d$ belongs to $A$ and maximizes $s_x$ over $A$.  In particular,
\begin{equation}\label{eq:raystaticvalues}
 h(\beta_t)=h_0+t r,\qquad U(p,z(\beta_t))=h_0.
\end{equation}
A self at $p$ can choose $\beta_t$ as successor.  Applying \eqref{eq:Bellmansupportidentity}, \eqref{eq:raystaticvalues}, and \eqref{eq:rayresidual} gives
\begin{align}
 S(p)&\geq C_{\phi,z}(\beta_t;p)\notag\\
 &=h_0+(1-\lambda)t r+\delta\{S(p)-R_t\}.\label{eq:contactdeviation}
\end{align}
Since $(1-\delta)S(p)=h_0$, this rearranges to
\[
 (1-\lambda)r\leq\delta\frac{R_t}{t}.
\]
By \cref{lem:rayerror}, $r\leq0$.  Thus $u_j(x)\leq h_0$ for every $x\in A$.  The index $j$ was arbitrary, so
\[
 M(x)\leq h_0=U(p,x)\leq M(x)\qquad(x\in A).
\]
This proves $M(x)=h_0$.  Finally,
\[
 0=M(x)-U(p,x)=\sum_{i\in J}p^i\{M(x)-u_i(x)\}
\]
is a sum of nonnegative terms.  If $p^i>0$, the corresponding bracket must be zero, proving the support assertion.
\end{proof}

At a contact state, every currently optimal action is already evaluated as highly as any pure worldview can evaluate it.  Otherwise, a sufficiently small change toward a more favorable evaluator would yield a first-order gain in experienced utility.  Convexity makes the loss associated with the supporting continuation smaller than first order along that change.  Inequality \eqref{eq:contactdeviation} expresses this comparison without assuming the equilibrium policy is continuous.

\subsection{Accumulation points and purification}
\begin{proof}[Proof of \cref{thm:genericpurification} for pure strategies]
By \cref{lem:convexvalue}, $(1-\delta)S-h$ is continuous.  Its zero set $\mathcal Z$ is therefore compact.  Along any equilibrium path, \cref{thm:alignment} gives
\begin{equation}\label{eq:contactpathgap}
 (1-\delta)S(\alpha_t)-h(\alpha_t)=A_t-h_t\longrightarrow0.
\end{equation}
Every path has an accumulation point because $\Delta^n$ is compact.  Continuity in \eqref{eq:contactpathgap} puts every such point in $\mathcal Z$ and proves nonemptiness.  The support characterization is \cref{lem:contactmax}.

Assume \eqref{eq:uniquevaluator}.  At any $p\in\mathcal Z$, choose $x\in z^*(p)$.  This action is implementable.  By \eqref{eq:contactmaxima}, the support of $p$ consists of maximizers of $u_j(x)$.  Their uniqueness forces $p$ to be pure.  Thus $\mathcal Z\subseteq\mathcal P$ and
\begin{equation}\label{eq:approachpuregeneric}
 \dist(\alpha_t,\mathcal P)\longrightarrow0.
\end{equation}
Indeed, failure of this limit would give a subsequence at a fixed positive distance from $\mathcal P$ and a nonpure accumulation point.

To obtain convergence to a single vertex, suppose that the increments do not tend to zero. There are $\epsilon>0$ and infinitely many dates $t$ with $\|\alpha_{t+1}-\alpha_t\|_1\geq\epsilon$. Compactness and finiteness of $X$ give a subsequence of these dates such that
\[
 \alpha_{t_r}\to p,\quad \alpha_{t_r+1}\to q,\quad p\ne q,
 \qquad x_{t_r+1}=x
\]
for a fixed action $x$, with $\|p-q\|_1\geq\epsilon$. Both $p$ and $q$ belong to $\mathcal Z\subseteq\mathcal P$.  Static optimality at $\alpha_{t_r+1}$ and continuity imply $x\in z^*(q)$ and $U(q,x)=L$.  The $k=1$ conclusion of \cref{thm:alignment} also gives $U(p,x)=L$.  By \cref{lem:contactmax}, $U(q,x)=M(x)$.  Hence two distinct pure worldviews maximize $u_j(x)$, contradicting \eqref{eq:uniquevaluator}.  We have proved
\begin{equation}\label{eq:vanishingincrements}
 \|\alpha_{t+1}-\alpha_t\|_1\longrightarrow0.
\end{equation}
Distinct simplex vertices have $\ell^1$ distance $2$.  By \eqref{eq:approachpuregeneric}, eventually every state is within $1/4$ of a vertex.  By \eqref{eq:vanishingincrements}, eventually consecutive states are less than $1$ apart.  A move between two different vertex neighborhoods would have distance at least $2-1/4-1/4=3/2$, so all sufficiently late states lie near one fixed vertex.  Equation \eqref{eq:approachpuregeneric} then proves convergence to that vertex.  The case $n=1$ is immediate.
\end{proof}

\begin{corollary}[Payoff ties are necessary for persistent mixing]\label{cor:genericrobustness}
If a stationary MPE path has a mixed accumulation point $p$, every $x\in z^*(p)$ has at least two maximizing pure evaluators, and all worldviews in the support of $p$ attain that maximum.  For fixed finite $J$ and $X$, the set of payoff arrays admitting a stationary MPE path that does not converge to a pure worldview has empty interior in $\R^{|J||X|}$ and is contained in
\begin{equation}\label{eq:hyperplanenecessity}
 \bigcup_{x\in X}\ \bigcup_{i<j}\{u:u_i(x)=u_j(x)\}.
\end{equation}
In particular, there is no nonempty open set of unrestricted payoff arrays on each of which a nonpurifying stationary equilibrium exists.
\end{corollary}
\begin{proof}
The first assertion follows from \cref{thm:genericpurification}(ii), because a mixed worldview has at least two positive coordinates.  If \eqref{eq:uniquevaluator} holds, part (iii) rules out every path that fails to converge to a pure worldview.  Failure of \eqref{eq:uniquevaluator} requires a tie between two maximizing pure evaluators of some implementable action, and hence membership in \eqref{eq:hyperplanenecessity}.  Each set in that finite union is a proper affine hyperplane and has empty interior; their finite union also has empty interior.  Equivalently, payoff arrays with all entries distinct within each action form an open dense subset of the full payoff space and satisfy \eqref{eq:uniquevaluator}.
\end{proof}

This restriction concerns all stationary equilibria, not only those satisfying the sufficient upper bound of \cref{thm:generaltrap}.  The common-value examples have exactly the kind of tie permitted by the corollary.  Payoff ties alone, however, do not guarantee persistence: the continuation incentives must also support it.

\begin{corollary}[Weak tailoring, conditional on existence]\label{cor:weaktailoringconvergence}
Suppose each implementable action has a pure worldview that weakly implements it and values it strictly more highly than every other pure worldview.  For every $0<\lambda,\delta<1$, every stationary MPE that exists converges from every initial state to a pure worldview.  The same conclusion holds almost surely under stationary randomization.
\end{corollary}
\begin{proof}
The strict cross-worldview comparison is \eqref{eq:uniquevaluator}.  Apply \cref{thm:genericpurification}; weak implementation is not needed for this asymptotic conclusion.  The randomized assertion follows from \cref{thm:randomcontact} below.
\end{proof}

\subsection{Finite-time adoption at low mindset flexibility}\label{sec:genericfinite}
The asymptotic result does not require a unique action at a pure worldview.  Such uniqueness does give finite-time adoption for sufficiently high $\lambda$, without requiring every implementable action to have a tailored worldview.

Assume \eqref{eq:uniquevaluator} and that $z^*(\alpha(j))=\{x_j\}$ for every $j\in J$.  Write $h_j=u_j(x_j)$ and define
\begin{equation}\label{eq:criticalpureindices}
 J^\dagger=\{j\in J:u_j(x_j)=M(x_j)\}.
\end{equation}
This set is nonempty: a pure worldview attaining $\max_{j,x}u_j(x)$ belongs to it.  For $j\in J^\dagger$ and $y\in X\setminus\{x_j\}$ put
\begin{equation}\label{eq:genericgaps}
 D_{jy}=h_j-u_j(y)>0,\qquad R_{jy}=M(y)-u_j(y)\geq0,
\end{equation}
and set
\begin{equation}\label{eq:genericfinitecutoff}
 \widehat\lambda=
 \max\left(\{0\}\cup
 \left\{1-\frac{D_{jy}}{R_{jy}}:
    j\in J^\dagger,\ y\in X\setminus\{x_j\},\ R_{jy}>0\right\}\right)<1.
\end{equation}

\begin{theorem}[Finite-time purification without full tailoring]\label{thm:genericfinite}
Under these assumptions, if $\widehat\lambda<\lambda<1$ and $0<\delta<1$, every stationary MPE that exists reaches an absorbing pure worldview in finite time from every initial state.  For stationary randomized MPE the hitting time is finite almost surely.
\end{theorem}
\begin{proof}
For $j\in J^\dagger$, write $p=\alpha(j)$ and $x=x_j$.  For every $y\ne x$,
\begin{equation}\label{eq:genericstrictcap}
 (1-\lambda)M(y)+\lambda u_j(y)
 =h_j-\{D_{jy}-(1-\lambda)R_{jy}\}<h_j.
\end{equation}
If $R_{jy}=0$, strictness follows from $D_{jy}>0$; otherwise it follows from \eqref{eq:genericfinitecutoff}.  For action $x$,
\[
 (1-\lambda)U(\beta,x)+\lambda u_j(x)\leq h_j,
\]
with equality only at $\beta=p$ by \eqref{eq:uniquevaluator}.  Let $S(p)=C_{\phi,z}(\phi(p);p)$ for the equilibrium under consideration. The postponement inequality and the upper bounds give
\[
 \frac{h_j}{1-\delta}\leq S(p)\leq\frac{h_j}{1-\delta}.
\]
Writing $(\beta_k,y_k)=(\phi^{k+1}(p),z(\phi^{k+1}(p)))$, equality implies
\[
 0=\sum_{k=0}^{\infty}\delta^k
 \bigl[h_j-(1-\lambda)U(\beta_k,y_k)-\lambda u_j(y_k)\bigr].
\]
Every summand is nonnegative, and the first is zero only for $(\beta_0,y_0)=(p,x)$. Thus $\phi(p)=p$.

The same argument gives a neighborhood on which every equilibrium chooses $p$ immediately.  Define
\[
 g_y(\alpha)=(1-\lambda)M(y)+\lambda U(\alpha,y),\qquad
 \gamma_{jy}=D_{jy}-(1-\lambda)R_{jy}>0,
\]
and $L_{xy}=\max_{i\in J}|u_i(x)-u_i(y)|>0$ for $y\ne x$.  When $X\setminus\{x\}\ne\varnothing$, choose
\begin{equation}\label{eq:localradius}
 r_j=\min\left\{1,\ \min_{y\ne x}\frac{\gamma_{jy}}{2\lambda L_{xy}}\right\}>0;
\end{equation}
if there is no other action, take $r_j=1$.  For $\|\alpha-p\|_1<r_j$ and $y\ne x$,
\[
 g_x(\alpha)-g_y(\alpha)
 \geq\gamma_{jy}-\lambda L_{xy}\|\alpha-p\|_1
 >\frac{\gamma_{jy}}2>0.
\]
For every future pair $(\beta,y)$,
\begin{equation}\label{eq:localcapturegap}
 \begin{split}
 g_x(\alpha)-q_\alpha(\beta,y)
 &=g_x(\alpha)-g_y(\alpha)
   +(1-\lambda)\{M(y)-U(\beta,y)\}\\
 &\geq0.
 \end{split}
\end{equation}
If $y\ne x$, the first difference is strictly positive. If $y=x$, the second difference vanishes only at $\beta=p$, since $p$ uniquely maximizes the evaluation of $x$. Hence equality in \eqref{eq:localcapturegap} holds only for $(\beta,y)=(p,x)$. Choosing the absorbing state $p$ attains $g_x(\alpha)/(1-\delta)$. For an equilibrium continuation $(\alpha_k,x_k)_{k\geq1}$ from $\alpha$, optimality therefore gives
\[
 0=\sum_{k=1}^\infty\delta^{k-1}
       [g_x(\alpha)-q_\alpha(\alpha_k,x_k)].
\]
Every term is nonnegative. The first vanishes only when $\alpha_1=p$ and $x_1=x$. Thus every equilibrium selects $p$ immediately throughout this neighborhood. The pointwise identity \eqref{eq:localcapturegap} also applies to randomized successors.

By \cref{thm:genericpurification}, an equilibrium path converges to a pure $p$.  Its contact-state property and unique static action imply that its index lies in $J^\dagger$.  The path eventually enters the corresponding neighborhood and reaches $p$ in the next period.  The randomized assertion follows from the almost-sure convergence and the first-period equality argument in \cref{thm:randomcontact}; its proof below includes the hitting-time argument.
\end{proof}

The strictly tailored theorem below is stronger in a different respect.  It supplies equilibrium existence and a global one-step policy characterization.  \Cref{thm:genericfinite} requires no tailored destination for every action, but allows a transient period and remains conditional on equilibrium existence.

\section{Strictly tailored worldviews and one-step convergence}\label{sec:tailored}

We prove Theorem~\ref{thm:onestepcharacterization} under the strictly tailored condition of \cref{def:strictlytailored}.  Recall the tailored pure worldview $\alpha^I(x)=\alpha(i(x))$, the differences $D_{xy},R_{xy}$ in \eqref{eq:Dxy}--\eqref{eq:Rxy}, and the threshold $\lambda^\star$ in \eqref{eq:lambdastar}.

For $\lambda>\lambda^\star$ and $x\neq y$, the inequality $(1-\lambda)R_{xy}<D_{xy}$ gives
\begin{align}
 &(1-\lambda)u_{i(y)}(y)+\lambda u_{i(x)}(y)\notag\\
 &\qquad=u_{i(x)}(y)+(1-\lambda)R_{xy}
 <u_{i(x)}(y)+D_{xy}=u_{i(x)}(x).\label{eq:anchorcomparison}
\end{align}
This sufficient threshold ensures absorption of every tailored pure worldview.

\subsection{Absorbing pure worldviews}

The next lemma adapts the absorption argument of \citet[Proposition~2, p.~729; Online Appendix, p.~3]{BBMZ2021} to the strictly tailored condition, using the explicit threshold \eqref{eq:lambdastar}.

\begin{lemma}[Absorption of tailored pure worldviews]\label{lem:anchorabsorption}
Assume \eqref{eq:anchorstatic}--\eqref{eq:anchorcross} and $\lambda^\star<\lambda<1$.  In every stationary MPE,
\begin{equation}\label{eq:anchorabsorbs}
 \phi(\tailored{x})=\tailored{x}
 \quad\text{for every }x\in X^I.
\end{equation}
\end{lemma}
\begin{proof}
Fix $x\in X^I$ and abbreviate $i=i(x)$ and $p=\pure{i}$.  For any candidate next worldview $\beta$ and any $y\in z^*(\beta)$, let
\[
 q=(1-\lambda)U(\beta,y)+\lambda u_i(y).
\]
If $y\neq x$, then $y\in X^I$ and \eqref{eq:anchorcross} for $y$ implies $U(\beta,y)\leq u_{i(y)}(y)$.  By \eqref{eq:anchorcomparison},
\begin{equation}\label{eq:anchorflowother}
 q\leq(1-\lambda)u_{i(y)}(y)+\lambda u_i(y)<u_i(x).
\end{equation}
If $y=x$, then
\begin{align}
 u_i(x)-q
 &=(1-\lambda)\bigl[u_i(x)-U(\beta,x)\bigr]\notag\\
 &=(1-\lambda)\sum_{j\neq i}\beta^j\bigl[u_i(x)-u_j(x)\bigr]\geq0.\label{eq:anchorflowsame}
\end{align}
Every coefficient in brackets is strictly positive, so equality holds if and only if $\beta=p$.

Suppose a stationary MPE has $\beta_0=\phi(p)\neq p$.  Put $\beta_k=\phi^k(\beta_0)$, $y_k=z(\beta_k)$, and
\[
 q_k=(1-\lambda)U(\beta_k,y_k)+\lambda u_i(y_k).
\]
Equations \eqref{eq:anchorflowother}--\eqref{eq:anchorflowsame} give $q_k\leq u_i(x)$ for every $k$, with $q_0<u_i(x)$.  Hence
\begin{align}
 \frac{u_i(x)}{1-\delta}-C_{\phi,z}(\beta_0;p)
 &=\sum_{k=0}^{\infty}\delta^k\bigl[u_i(x)-q_k\bigr]\notag\\
 &\geq u_i(x)-q_0>0.\label{eq:anchorCstrict}
\end{align}
But $h(p)=u_i(x)$, and \cref{lem:postponement} requires
\[
 C_{\phi,z}(\phi(p);p)\geq\frac{u_i(x)}{1-\delta}.
\]
This contradicts \eqref{eq:anchorCstrict}, proving absorption.
\end{proof}

\subsection{Existence and a value bound}

For $g_x,G,\Gamma$ from \eqref{eq:gxanchor}--\eqref{eq:Ganchor}, consider a policy that chooses a maximizing tailored pure worldview.

\begin{theorem}[Construction of a stationary MPE]\label{thm:anchorexistence}
Assume the strictly tailored condition and $\lambda^\star<\lambda<1$.  Choose any selections
\[
 z(\alpha)\in z^*(\alpha),\qquad x^\dagger(\alpha)\in\Gamma(\alpha),
\]
and set
\begin{equation}\label{eq:anchorpolicy}
 \phi(\alpha)=\tailored{x^\dagger(\alpha)}.
\end{equation}
Then $(\phi,z)$ is a stationary MPE for every $\delta\in(0,1)$.
\end{theorem}
\begin{proof}
At $p=\tailored{x}$,
\[
 g_x(p)=u_{i(x)}(x).
\]
For $y\neq x$, \eqref{eq:anchorcomparison} gives $g_y(p)<g_x(p)$.  Therefore $\Gamma(p)=\{x\}$ and $\phi(p)=p$.  By \eqref{eq:anchorstatic}, $z(p)=x$ as well.

Fix a current worldview $\alpha$ and an arbitrary successor $\beta$.  At every date on its continuation path, put $\gamma_k=\phi^k(\beta)$ and $y_k=z(\gamma_k)\in X^I$.  The strictly tailored condition gives
\begin{align}
 &(1-\lambda)U(\gamma_k,y_k)+\lambda U(\alpha,y_k)\notag\\
 &\qquad\leq(1-\lambda)u_{i(y_k)}(y_k)+\lambda U(\alpha,y_k)
 =g_{y_k}(\alpha)\leq G(\alpha).\label{eq:firstflowanchorG}
\end{align}
Summing yields
\begin{equation}\label{eq:Vanchorupper}
 C_{\phi,z}(\beta;\alpha)\leq\frac{G(\alpha)}{1-\delta}.
\end{equation}
Choosing $\beta=\tailored{x^\dagger(\alpha)}$ reaches an absorbing state and gives the constant one-period assessment $g_{x^\dagger(\alpha)}(\alpha)=G(\alpha)$.  It attains the bound in \eqref{eq:Vanchorupper}.  This proves \eqref{eq:MPEworldview}, while \eqref{eq:MPEaction} holds by construction.
\end{proof}

\subsection{Characterization of all stationary equilibria}

\begin{proof}[Proof of \cref{thm:onestepcharacterization}]
Sufficiency of \eqref{eq:fullcharacterization} is \cref{thm:anchorexistence}.  For necessity, take an arbitrary stationary MPE and a current state $\alpha$.  By \cref{lem:anchorabsorption}, every $\tailored{x}$ is absorbing in this equilibrium.  Therefore the consumer can obtain $g_x(\alpha)/(1-\delta)$ by selecting $\tailored{x}$, and
\begin{equation}\label{eq:currentGfloor}
 C_{\phi,z}(\phi(\alpha);\alpha)\geq\frac{G(\alpha)}{1-\delta}.
\end{equation}
The bound \eqref{eq:firstflowanchorG} uses only static optimality and \eqref{eq:anchorcross}; it is valid at every future state of this arbitrary equilibrium.  Thus \eqref{eq:Vanchorupper} gives the reverse inequality, proving \eqref{eq:equilibriumvalue}.

Set $\beta=\phi(\alpha)$, $\gamma_k=\phi^k(\beta)$, and $y_k=z(\gamma_k)$.  Subtract the equilibrium value from its bound and expand each term:
\begin{align}
 0
 &=\frac{G(\alpha)}{1-\delta}-C_{\phi,z}(\beta;\alpha)\notag\\
 &=\sum_{k=0}^{\infty}\delta^k
   \Bigl(G(\alpha)-g_{y_k}(\alpha)
   +(1-\lambda)\bigl[u_{i(y_k)}(y_k)-U(\gamma_k,y_k)\bigr]\Bigr).
 \label{eq:tailoredzerosum}
\end{align}
Both differences inside every summand are nonnegative, and all weights $\delta^k$ and $1-\lambda$ are strictly positive.  Their sum can be zero only if both differences vanish at every date.  In particular, at $k=0$, writing $y=y_0=z(\beta)$, we obtain
\begin{equation}\label{eq:tailoredfirstequalities}
 g_y(\alpha)=G(\alpha),\qquad U(\beta,y)=u_{i(y)}(y).
\end{equation}
The second equality is equivalent to
\[
 0=u_{i(y)}(y)-U(\beta,y)
   =\sum_{j\neq i(y)}\beta^j\bigl[u_{i(y)}(y)-u_j(y)\bigr].
\]
All brackets are strictly positive.  Hence $\beta^j=0$ for $j\neq i(y)$ and $\beta^{i(y)}=1$, so $\beta=\tailored{y}$.  The first equality in \eqref{eq:tailoredfirstequalities} gives $y\in\Gamma(\alpha)$.  This proves necessity.  Absorption and the unique action at the chosen tailored worldview give \eqref{eq:finalanchorpath}.
\end{proof}

\begin{corollary}[Indifference sets and the discount factor]\label{cor:policycells}
Under the assumptions of \cref{thm:onestepcharacterization}, the set of stationary MPE is independent of $\delta\in(0,1)$.  The worldview policy is uniquely determined outside a finite union of proper affine hyperplanes in the affine hull of $\Delta^n$.  In particular, it is unique at almost every initial worldview, relative to $(n-1)$-dimensional Lebesgue measure when $n\geq2$.
\end{corollary}
\begin{proof}
The conditions \eqref{eq:fullcharacterization} contain no $\delta$.  For $x\neq y$, the difference $g_x-g_y$ is affine and strictly positive at $\tailored{x}$ by \eqref{eq:anchorcomparison}.  It is therefore not the zero function.  At $\tailored{y}$ it is strictly negative.  Its zero set is a proper affine hyperplane in the affine hull of the simplex.  Outside the finite union of these zero sets, all the $g_x(\alpha)$ are distinct, so $\Gamma(\alpha)$ is a singleton and \eqref{eq:fullcharacterization} determines $\phi(\alpha)$ uniquely.  Proper affine hyperplanes have zero relative Lebesgue measure.  If $|X^I|=1$, the policy is unique everywhere and the union is empty.
\end{proof}

The weak choice regions are the polytopes
\begin{equation}\label{eq:policyregions}
 R_x=\{\alpha\in\Delta^n:g_x(\alpha)\geq g_y(\alpha)
                                \text{ for every }y\in X^I\}.
\end{equation}
Where $x$ is the unique maximizer, every stationary MPE sends the consumer to $\alpha^I(x)$.  On boundaries shared by several regions, any of the corresponding tailored pure worldviews can be chosen.  Thus the strictly tailored condition gives a multi-action counterpart of the one-step policies in \citet[Proposition~1, p.~728]{BBMZ2021}.  Every such policy satisfies $\phi\circ\phi=\phi$: after the first change, the worldview remains fixed.

\begin{remark}[The economic content of the threshold]\label{rem:econthreshold}
At $\alpha^I(x)$, $D_{xy}$ is the loss from taking action $y$ instead of $x$, both evaluated under the current worldview.  The quantity $R_{xy}$ is the gain in the evaluation of $y$ that would result from adopting its tailored worldview.  Therefore
\[
 u_{i(x)}(x)-g_y(\alpha^I(x))
 =D_{xy}-(1-\lambda)R_{xy}>0
\]
means that the prospective gain in experienced utility does not offset the current objection to switching actions.  This makes the tailored pure worldviews absorbing.  Since every implementable action has such a pure worldview, its maximal experienced utility can then be obtained by a credible one-period change in worldview.
\end{remark}

\begin{remark}[Convergence without strictly tailored worldviews]\label{rem:notnecessary}
Strictly tailored worldviews are not necessary for convergence.  In Case~2 of \citet[Proposition~3, pp.~730--731]{BBMZ2021}, worldview~2 happiness-dominates worldview~1: in particular, $u_2(1)>u_1(1)$ while worldview~2 uniquely chooses action~2.  Thus action~1 does not satisfy \eqref{eq:anchorstatic}--\eqref{eq:anchorcross}, yet consumers eventually adopt worldview~2.  Here convergence proceeds through changes in both the action and its evaluation rather than direct selection of a pure worldview tailored to the current implementable action.
\end{remark}

\section{Weak implementation and equilibrium existence}\label{sec:weaktailoring}

In \eqref{eq:weakcountermatrix}, the pure worldview that evaluates both actions most favorably is indifferent between them. A deterministic action choice at that state leaves some selves seeking an unattainable continuation value.

\begin{proof}[Proof of \cref{thm:weaktailoringnonexistence}]
Suppose such an equilibrium $(\phi,z)$ exists.  Write $p=\alpha(1)$.  For both actions, $u_1(x)=1>u_j(x)$ whenever $j\neq1$.  It follows that
\begin{equation}\label{eq:weakmaximum}
 h(\beta)\leq1,\qquad h(\beta)=1\iff\beta=p.
\end{equation}
At the current state $p$, the assessment of any future pair $(\beta,x)$ is
\[
 (1-\lambda)U(\beta,x)+\lambda
 \leq1,
\]
with equality only when $\beta=p$.  Postponement at $p$ gives a continuation value at least $1/(1-\delta)$, while the preceding bound gives the reverse inequality.  Equality in the first future period forces $\phi(p)=p$.

Call the action selected at $p$ action $a$, relabeling the two actions if necessary, and call the other action $b$.  Both strict preference regions remain nonempty: before relabeling, worldview~2 strictly favors $a$ and worldview~3 strictly favors $b$.  Define
\[
 d(\alpha)=U(\alpha,b)-U(\alpha,a),\qquad
 S_b=\{\alpha:d(\alpha)>0\}.
\]
Thus $S_b\neq\varnothing$, $d(p)=0$, and $z(p)=a$.

For a candidate successor $\beta$, let
\begin{align}
 D(\beta)&=\sum_{k=0}^\infty\delta^k
          [1-h(\phi^k(\beta))],\label{eq:weakD}\\
 N(\beta)&=\sum_{k=0}^\infty\delta^k
          \mathbf1_{\{z(\phi^k(\beta))=b\}}.\label{eq:weakN}
\end{align}
All terms are bounded and both series converge.  In particular,
\begin{equation}\label{eq:weakbounds}
 D(p)=N(p)=0,\quad D(\beta)\geq1-h(\beta)>0\ (\beta\neq p),\quad
 0\leq N(\beta)\leq\frac1{1-\delta}.
\end{equation}
For the equilibrium path starting at $\beta$, \eqref{eq:Hnextgeht} and \eqref{eq:Hrecursion} give $H_0\geq h(\beta)$.  Therefore
\begin{equation}\label{eq:weakDsmall}
 0\leq D(\beta)=\frac{1-H_0}{1-\delta}
 \leq\frac{1-h(\beta)}{1-\delta}.
\end{equation}

For each $k\geq0$, the current-worldview evaluation can be written as
\[
 U(\alpha,z(\phi^k(\beta)))
 =U(\alpha,a)+d(\alpha)\mathbf1_{\{z(\phi^k(\beta))=b\}}.
\]
The experienced component satisfies
\[
 \sum_{k=0}^\infty\delta^k h(\phi^k(\beta))
 =\frac1{1-\delta}-D(\beta).
\]
Substituting both identities into \eqref{eq:continuationV} gives
\begin{equation}\label{eq:weakdecomposition}
 C_{\phi,z}(\beta;\alpha)
 =\frac{(1-\lambda)+\lambda U(\alpha,a)}{1-\delta}
       +\lambda d(\alpha)N(\beta)-(1-\lambda)D(\beta).
\end{equation}
If $d(\alpha)\leq0$, the last two terms are nonpositive for every $\beta$, equal zero at $p$, and have a strictly negative sum whenever $\beta\ne p$, because $D(\beta)>0$. Thus $p$ is the unique maximizing successor and
\begin{equation}\label{eq:weakties}
 d(\alpha)\leq0\quad\Longrightarrow\quad\phi(\alpha)=p.
\end{equation}
For a current state $\alpha\in S_b$, write the same decomposition as
\begin{equation}\label{eq:weakobjective}
 \begin{split}
 C_{\phi,z}(\beta;\alpha)
 &=\frac{(1-\lambda)+\lambda U(\alpha,a)}{1-\delta}
   +W_\alpha(\beta),\\
 W_\alpha(\beta)
 &=\lambda d(\alpha)N(\beta)-(1-\lambda)D(\beta).
 \end{split}
\end{equation}
Choose $q\in S_b$ and set $\beta_\varepsilon=(1-\varepsilon)p+\varepsilon q$ for $0<\varepsilon<1$.  Since both utilities at $p$ equal one,
\[
 d(\beta_\varepsilon)=\varepsilon d(q)>0,\qquad
 h(\beta_\varepsilon)=1-\varepsilon[1-h(q)]\longrightarrow1.
\]
Thus $N(\beta_\varepsilon)\geq1$ and, by~\eqref{eq:weakDsmall}, $D(\beta_\varepsilon)\to0$.  Equilibrium optimality and~\eqref{eq:weakobjective} imply
\begin{equation}\label{eq:weakfirstlower}
 W_\alpha(\phi(\alpha))
 \geq \liminf_{\varepsilon\downarrow0}W_\alpha(\beta_\varepsilon)
 \geq\lambda d(\alpha)>0.
\end{equation}

We next prove $\phi(S_b)\subseteq S_b$. Put $\beta=\phi(\alpha)$ with $\alpha\in S_b$. If $d(\beta)\leq0$, \eqref{eq:weakties} gives $\phi(\beta)=p$. If $z(\beta)=a$, no action $b$ is ever chosen from $\beta$, so $N(\beta)=0$ and $W_\alpha(\beta)\leq0$, contrary to \eqref{eq:weakfirstlower}. If $z(\beta)=b$, static optimality requires $d(\beta)=0$ and $\beta\ne p$. Only the initial action is $b$, so $N(\beta)=1$ and $D(\beta)>0$. This gives $W_\alpha(\beta)<\lambda d(\alpha)$, again contradicting \eqref{eq:weakfirstlower}. Thus $d(\beta)>0$ and $\phi(S_b)\subseteq S_b$.

It follows that every path starting in $S_b$ chooses $b$ forever, so
\[
 N(\beta)=\frac1{1-\delta}\qquad(\beta\in S_b).
\]
In particular,
\[
 W_\alpha(\beta_\varepsilon)
 \longrightarrow\frac{\lambda d(\alpha)}{1-\delta}.
\]
But no successor attains this value.  At $p$ the value is zero, and at every $\beta\neq p$,~\eqref{eq:weakbounds} implies
\[
 W_\alpha(\beta)
 \leq\frac{\lambda d(\alpha)}{1-\delta}-(1-\lambda)D(\beta)
 <\frac{\lambda d(\alpha)}{1-\delta}.
\]
This contradicts the requirement that $\phi(\alpha)$ maximize the continuation payoff.  Hence no stationary pure-strategy MPE exists.
\end{proof}

The example separates convergence from existence. Both actions have the same unique maximizing evaluator, so any existing stationary equilibrium would satisfy the convergence conclusion of \cref{thm:genericpurification}. A pure action policy at that evaluator must nevertheless select only one of the tied actions. A self preferring the other action can approach the highest experienced utility without attaining its best continuation value. Weak implementation is therefore insufficient for the existence conclusion of \cref{thm:onestepcharacterization}. The payoff array also satisfies cross-worldview distinctness in \eqref{eq:genericity}: its only tie, $u_1(a)=u_1(b)=1$, is between actions within one worldview. This distinction concerns the sophisticated consumer's stationary game, rather than the na\"ive consumer studied in \citet[Proposition~4, p.~732]{BBMZ2021}.

\section{Stationary randomized strategies}\label{sec:randomization}

A mixed worldview is a vector of utility weights, not a randomized strategy.  The preceding sections use deterministic choices of that vector.  We now allow each self to randomize over both its current action and its next worldview, while retaining stationarity and Markov dependence on the current worldview.

\subsection{The equilibrium condition}

A stationary randomized strategy is a Borel probability kernel $\Pi_\alpha$ on $X\times\Delta^n$.  At state $\alpha_t$, a pair $(x_t,\alpha_{t+1})$ is drawn from $\Pi_{\alpha_t}$.  Write $s_\alpha$ and $\kappa_\alpha$ for its action and successor-state marginals, respectively, and define the transition operator
\[
 (Tf)(\alpha)=\int f(\beta)\,\kappa_\alpha(d\beta).
\]
For each fixed initial state $\alpha_0$, the kernel specifies the law of the sequence by
\[
 \mathbb P\bigl((x_t,\alpha_{t+1})\in B\mid\mathcal F_t\bigr)
 =\Pi_{\alpha_t}(B),\qquad
 \mathcal F_t=\sigma(\alpha_0,x_0,\alpha_1,\ldots,x_{t-1},\alpha_t),
\]
for every Borel set $B\subseteq X\times\Delta^n$. All probabilities below refer to this law. Correlation between the two choices is allowed. The immediate utility depends on the current action, and the remaining continuation depends on the successor worldview, so the expected payoff depends only on the two marginals.

Let $\mathbb E_\beta$ denote expectation under this strategy with initial state $\alpha_0=\beta$.  The evaluation of a deterministic successor $\beta$ by the current worldview $\alpha$ is
\begin{equation}\label{eq:randomC}
 C_\Pi(\beta;\alpha)
 =\mathbb E_\beta\sum_{k=0}^\infty\delta^k
 \bigl[(1-\lambda)U(\alpha_k,x_k)+\lambda U(\alpha,x_k)\bigr].
\end{equation}
The dependence on the two worldview arguments is explicit. Define
\begin{align}
 v(\beta)&=(1-\lambda)\mathbb E_\beta
             \sum_{k=0}^\infty\delta^k U(\alpha_k,x_k),
 &w_j(\beta)&=\lambda\mathbb E_\beta
             \sum_{k=0}^\infty\delta^k u_j(x_k).
 \label{eq:randomcoefficients}
\end{align}
Successive integration against the Borel kernel makes each finite-horizon coefficient Borel in $\beta$. The uniform discounted tail bound implies that $v$ and every $w_j$ are bounded Borel functions. Moreover,
\begin{equation}\label{eq:randomjointborel}
 C_\Pi(\beta;\alpha)=v(\beta)+\sum_j\alpha^j w_j(\beta),
\end{equation}
so $C_\Pi$ is jointly Borel in $(\beta,\alpha)$ and affine in $\alpha$. The same tail bound permits expectations and discounted sums to be interchanged. Define the equilibrium continuation
\begin{equation}\label{eq:randomS}
 S(\alpha)=\int C_\Pi(\beta;\alpha)\,\kappa_\alpha(d\beta).
\end{equation}
A stationary randomized MPE satisfies, at every $\alpha$,
\begin{equation}\label{eq:randomMPE}
 s_\alpha(z^*(\alpha))=1,\qquad
 S(\alpha)\geq C_\Pi(\gamma;\alpha)\quad\text{for every }\gamma\in\Delta^n.
\end{equation}
These conditions are necessary and sufficient for optimality against any joint lottery $\mu$ over the two current choices. Indeed, the deviating self's payoff is
\[
 \int U(\alpha,x)\,\mu_X(dx)
 +\delta\int C_\Pi(\beta;\alpha)\,\mu_\Delta(d\beta)
 \leq h(\alpha)+\delta S(\alpha),
\]
where $\mu_X$ and $\mu_\Delta$ are its marginals; equality holds for $\mu=\Pi_\alpha$. Conversely, a strict improvement in either marginal would give a strict improvement in the joint choice. The successor distribution is concentrated on maximizing continuations. Indeed, \eqref{eq:randomS}--\eqref{eq:randomMPE} give
\[
 0=\int [S(\alpha)-C_\Pi(\beta;\alpha)]\,\kappa_\alpha(d\beta),
 \qquad S(\alpha)-C_\Pi(\beta;\alpha)\geq0.
\]
Thus
\begin{equation}\label{eq:randomargmaxmass}
 \kappa_\alpha\bigl(\{\beta:C_\Pi(\beta;\alpha)=S(\alpha)\}\bigr)=1.
\end{equation}
The maximizing set is Borel by \eqref{eq:randomjointborel}, but need not be closed. Equation~\eqref{eq:randomargmaxmass} specifies its probability, rather than the topological support of the distribution. In particular, a lottery cannot attain a continuation supremum when no deterministic successor attains it.

Static optimality implies $U(\alpha_t,x_t)=h(\alpha_t)$ simultaneously at all integer dates on a probability-one event. Choosing the current worldview again yields
\begin{equation}\label{eq:randompostpone}
 C_\Pi(\alpha;\alpha)=h(\alpha)+\delta S(\alpha),\qquad
 S(\alpha)\geq\frac{h(\alpha)}{1-\delta}.
\end{equation}
The identity follows by conditioning after the first future period; the inequality follows from~\eqref{eq:randomMPE}.

\subsection{Persistence and one-step adoption}

\begin{theorem}[Persistence under stationary randomization]\label{thm:randompersistence}
Under~\eqref{eq:flowceiling}, every stationary randomized MPE starting at $\bar\alpha$ satisfies
\[
 q_{\bar\alpha}(\alpha_t,x_t)=h(\bar\alpha),\qquad
 \alpha_t\in K(\bar\alpha)\quad(t\geq1)
\]
simultaneously with probability one.  If $K(\bar\alpha)\cap\pures=\varnothing$, the distance bound in Theorem~\ref{thm:generaltrap} holds almost surely at every future date.

The conclusions of Theorems~\ref{thm:phasetransition} and~\ref{thm:manyworldviews} extend to stationary randomized MPE, with their pathwise assertions interpreted almost surely.
\end{theorem}
\begin{proof}
The pointwise upper bound gives $S(\bar\alpha)\leq h(\bar\alpha)/(1-\delta)$;~\eqref{eq:randompostpone} gives equality.  Hence
\[
 0=\mathbb E_{\bar\alpha}\sum_{t=1}^\infty\delta^{t-1}
       [h(\bar\alpha)-q_{\bar\alpha}(\alpha_t,x_t)].
\]
Each term is nonnegative almost surely by static optimality and~\eqref{eq:flowceiling}.  Tonelli's theorem and zero expectation force each term to vanish almost surely.  Intersecting these countably many probability-one events proves the simultaneous assertion.  The compact-set distance bound is deterministic and therefore applies on that event.

For the finite symmetric families, the same argument at any pure state, now with the bound $1$, implies that its successor is that same pure state almost surely.  Its action is uniquely determined.  At the balanced state, when $q_\lambda=1-\lambda/\lambda_c>0$, any pure destination attains $q_\lambda/(1-\delta)$.  Equality in the first-period upper bound forces the successor to be pure almost surely; it then remains at the realized pure destination.  When $q_\lambda<0$, only compromise outcomes attain zero, and the first part of the theorem applies.  For an arbitrary initial $s\in\mathcal T_\lambda$, the inequalities $h(s)=0$ and $\max_i g_i(s)<0$ give the same conclusion: every noncompromise pair has strictly negative assessment, so the general assertion forces $c$ at every date almost surely. At $\lambda=\lambda_c$, only compromise pairs and pure extreme pairs attain the bound. The pure-strategy constructions also satisfy the randomized best-response inequalities by integration and provide both critical continuations.  Taking $n=2$, $M=3$ recovers Theorem~\ref{thm:phasetransition}.
\end{proof}

\begin{theorem}[Randomized equilibrium under strictly tailored worldviews]\label{thm:randomtailored}
Under the assumptions of Theorem~\ref{thm:onestepcharacterization}, a Borel stationary randomized strategy is an MPE if and only if
\begin{equation}\label{eq:randomtailoredcondition}
 s_\alpha(z^*(\alpha))=1,\qquad
 \kappa_\alpha\bigl(\{\alpha^I(x):x\in\Gamma(\alpha)\}\bigr)=1
 \quad\text{for every }\alpha.
\end{equation}
For every such equilibrium, $S(\alpha)=G(\alpha)/(1-\delta)$.  From every initial state the consumer reaches a pure worldview in one period and remains there almost surely.  Randomization over future worldviews can occur only among the maximizing pure destinations in~\eqref{eq:randomtailoredcondition}.
\end{theorem}
\begin{proof}
At $p=\alpha^I(x)$, inequalities~\eqref{eq:anchorflowother}--\eqref{eq:anchorflowsame} bound every future assessment by $h(p)$, with equality only when the future pair is $(p,x)$.  Equations~\eqref{eq:randompostpone} and~\eqref{eq:randomMPE} force equality of the expected discounted sum.  The nonnegative first-period gap is zero almost surely, so $\kappa_p(\{p\})=1$.  Static uniqueness gives $s_p(\{x\})=1$.

For a general $\alpha$, the upper bound~\eqref{eq:firstflowanchorG} applies to every realized future pair.  Thus $S(\alpha)\leq G(\alpha)/(1-\delta)$.  Each maximizing pure destination is absorbing by the preceding paragraph and attains this bound, so equality holds.  Expanding the difference gives
\[
 0=\mathbb E_\alpha\sum_{t=1}^\infty\delta^{t-1}
 \left[G(\alpha)-g_{x_t}(\alpha)
 +(1-\lambda)\{u_{i(x_t)}(x_t)-U(\alpha_t,x_t)\}\right].
\]
Every expression inside braces and every score difference is nonnegative.  At $t=1$ both vanish almost surely.  Strict cross-worldview comparisons imply $\alpha_1=\alpha^I(x_1)$ and $x_1\in\Gamma(\alpha)$, proving necessity.

Conversely,~\eqref{eq:randomtailoredcondition} forces absorption at each tailored pure state because $\Gamma(\alpha^I(x))=\{x\}$.  Every selected successor then attains the same value $G(\alpha)/(1-\delta)$.  No deterministic or randomized deviation exceeds the pointwise upper bound, proving sufficiency.
\end{proof}

\subsection{Almost-sure convergence of experienced utility}

For a stationary randomized MPE define state functions
\begin{align}
 H(\alpha)&=(1-\delta)\mathbb E_\alpha\sum_{k=0}^\infty\delta^k h(\alpha_k),\label{eq:randomH}\\
 Q(\alpha)&=(1-\delta)\mathbb E_\alpha\sum_{k=1}^\infty\delta^{k-1}U(\alpha,x_k),\label{eq:randomQ}\\
 A(\alpha)&=(1-\delta)S(\alpha)=(1-\lambda)(TH)(\alpha)+\lambda Q(\alpha).
 \label{eq:randomA}
\end{align}
All are bounded Borel functions.  Here $H(\alpha_t)$ is a conditional expectation of discounted future experienced utility, rather than a discounted sum evaluated on a single realized future path.

\begin{theorem}[Utility convergence under stationary randomization]\label{thm:randomalignment}
For every fixed initial state, $H(\alpha_t)$ is a bounded submartingale with respect to the history observed upon arrival at date $t$.  There is a bounded random variable $L$ such that
\begin{equation}\label{eq:randomlimits}
 H(\alpha_t),\ h(\alpha_t),\ A(\alpha_t),\ Q(\alpha_t)\longrightarrow L
 \quad\text{almost surely and in }L^1.
\end{equation}
For each fixed integer $k\geq1$,
\begin{equation}\label{eq:randomlossbound}
 \mathbb E_{\alpha_0}\sum_{t=0}^\infty
 [h(\alpha_t)-U(\alpha_t,x_{t+k})]
 \leq\frac{(1-\lambda)(\mathbb E_{\alpha_0}L-H(\alpha_0))}
 {\lambda(1-\delta)^2\delta^{k-1}}.
\end{equation}
In particular, each of these nonnegative loss sequences is summable almost surely.
\end{theorem}
\begin{proof}
Postponement and static optimality give $A\geq h$ and $Q\leq h$, respectively.  From~\eqref{eq:randomA},
\[
 h\leq(1-\lambda)TH+\lambda Q\leq(1-\lambda)TH+\lambda h,
 \qquad\text{hence }TH\geq h.
\]
Conditioning the series for $H$ after its first period yields
\begin{equation}\label{eq:randomdrift}
 H=(1-\delta)h+\delta TH,\qquad
 d:=TH-H=(1-\delta)(TH-h)\geq0.
\end{equation}
If $\mathcal F_t$ is the history through $\alpha_t$ but before the date-$t$ draw, then
\[
 \mathbb E[H(\alpha_{t+1})\mid\mathcal F_t]
 =(TH)(\alpha_t)=H(\alpha_t)+d(\alpha_t).
\]
Since $\underline u\leq H\leq\overline u$, the submartingale convergence theorem \citep[Theorem~4.2.11]{Durrett2019} gives an almost-sure limit $L$.  The deterministic bounds pass to this limit, and bounded convergence gives convergence in $L^1$.  Moreover, telescoping expectations gives
\begin{equation}\label{eq:randomdriftsum}
 \mathbb E\sum_{t=0}^{T}d(\alpha_t)
 =\mathbb E H(\alpha_{T+1})-H(\alpha_0)
 \leq\overline u-H(\alpha_0).
\end{equation}
Monotone convergence implies
\[
 \mathbb E\sum_{t=0}^\infty d(\alpha_t)
 =\mathbb E L-H(\alpha_0)<\infty.
\]
Thus $\sum_t d(\alpha_t)<\infty$ almost surely, and $d(\alpha_t)\to0$.  Rearranging~\eqref{eq:randomdrift},
\[
 h=H-\frac{\delta}{1-\delta}d,\qquad TH=H+d,
\]
so both $h(\alpha_t)$ and $(TH)(\alpha_t)$ tend to $L$.  The inequalities
\[
 h\leq A\leq(1-\lambda)TH+\lambda h,
 \qquad 0\leq\lambda(h-Q)\leq(1-\lambda)(TH-h)
\]
then prove the remaining almost-sure limits.  Boundedness gives all the $L^1$ limits.

Set $b_k(\alpha)=h(\alpha)-\mathbb E_\alpha U(\alpha,x_k)\geq0$.  By bounded convergence for the discounted series and nonnegativity,
\[
 (1-\delta)\delta^{k-1}b_k(\alpha)
 \leq h(\alpha)-Q(\alpha)
 \leq\frac{1-\lambda}{\lambda(1-\delta)}d(\alpha).
\]
Consequently, with $c_k=(1-\lambda)/[\lambda(1-\delta)^2\delta^{k-1}]$,
\[
 b_k(\alpha)\leq c_k d(\alpha),\qquad
 \mathbb E[h(\alpha_t)-U(\alpha_t,x_{t+k})\mid\mathcal F_t]
 =b_k(\alpha_t).
\]
Apply Tonelli's theorem and the preceding conditional expectation identity to obtain
\[
 \mathbb E\sum_{t=0}^\infty[h(\alpha_t)-U(\alpha_t,x_{t+k})]
 =\sum_{t=0}^\infty\mathbb E b_k(\alpha_t)
 \leq c_k\,\mathbb E\sum_{t=0}^\infty d(\alpha_t),
\]
which is~\eqref{eq:randomlossbound}.  A nonnegative random variable with finite expectation is finite almost surely.  Since the set of fixed integer horizons is countable, the conclusions can be imposed simultaneously for every $k\geq1$.
\end{proof}

Under randomization, conditional discounted utility satisfies a submartingale inequality in place of pathwise monotonicity. The argument uses stationarity to reproduce the continuation after a one-period postponement.

\subsection{Limiting support and generic purification}\label{sec:randomcontact}
For a stationary randomized MPE, $S$ is the equilibrium continuation in \eqref{eq:randomS}. By \eqref{eq:randomargmaxmass}, for every $\alpha$ there is a deterministic successor attaining $S(\alpha)=\sup_\gamma C_\Pi(\gamma;\alpha)$. Set
\[
 \mathcal Z_\Pi=\{\alpha:(1-\delta)S(\alpha)=h(\alpha)\}.
\]

\begin{theorem}[Support of randomized-equilibrium limits]\label{thm:randomcontact}
For every stationary randomized MPE and every fixed initial state, every accumulation point of the realized worldview path belongs to $\mathcal Z_\Pi$ almost surely.  At each $p\in\mathcal Z_\Pi$, \eqref{eq:contactmaxima} holds for every $x\in z^*(p)$.  Under \eqref{eq:uniquevaluator}, there is a random index $J_\infty\in J$ such that
\begin{equation}\label{eq:randompurelimit}
 \alpha_t\longrightarrow\alpha(J_\infty)\qquad\text{almost surely}.
\end{equation}
For fixed finite $J,X$, the payoff arrays admitting a stationary randomized MPE that fails to converge to a pure worldview with positive probability from some fixed initial state are contained in the finite union of hyperplanes in \eqref{eq:hyperplanenecessity}. In particular, that set has empty interior.  Under the additional assumptions of \cref{thm:genericfinite}, an absorbing pure worldview is reached in finite time almost surely.
\end{theorem}
\begin{proof}
For fixed $\gamma$, the function $\alpha\mapsto C_\Pi(\gamma;\alpha)$ is affine, with slope bounded in $\ell^\infty$ by $\lambda B/(1-\delta)$.  The equilibrium inequalities in \eqref{eq:randomMPE} and integration against $\kappa_\alpha$ show that
\[
 S(\alpha)=\sup_\gamma C_\Pi(\gamma;\alpha).
\]
Indeed, the inequality $S\geq C_\Pi(\gamma;\alpha)$ gives one direction and the fact that $S$ is an average of these values gives the other.  Thus \eqref{eq:Slipschitz} and convexity hold.  For each $\beta$ define
\[
 f_\beta(\alpha)=\int C_\Pi(\gamma;\alpha)\,\kappa_\beta(d\gamma).
\]
For all $\alpha$, the equilibrium inequalities give $f_\beta(\alpha)\leq S(\alpha)$, and \eqref{eq:randomS} gives equality at $\alpha=\beta$. More explicitly,
\[
 f_\beta(\alpha)=S(\beta)+b_\beta\cdot(\alpha-\beta),\qquad
 b_\beta^j=\lambda\int\mathbb E_\gamma
       \sum_{k=0}^{\infty}\delta^k u_j(x_k)\,\kappa_\beta(d\gamma).
\]
The uniform series bound gives $\|b_\beta\|_\infty\leq\lambda B/(1-\delta)$. Thus these are supporting affine functions of exactly the kind used in \cref{lem:rayerror}.  Write $\overline U_\beta(\alpha)=\sum_x s_\beta(x)U(\alpha,x)$.  Conditioning after the initial draw at $\beta$ gives
\begin{equation}\label{eq:randomsupportidentity}
 C_\Pi(\beta;\alpha)
 =(1-\lambda)h(\beta)+\lambda\overline U_\beta(\alpha)
       +\delta f_\beta(\alpha).
\end{equation}
The identity uses only the marginals of the joint action--successor draw: the immediate term depends on the action and the remaining terms on the successor.  Correlation therefore causes no extra term.

Fix $p\in\mathcal Z_\Pi$ and a feasible ray $\beta_t=p+t\{\alpha(j)-p\}$.  For sufficiently small $t>0$, every action in $z^*(\beta_t)$ belongs to $A=z^*(p)$ and maximizes the directional slope over $A$, by the gap calculation in \cref{lem:contactmax}.  Since $s_{\beta_t}$ is supported on $z^*(\beta_t)$,
\[
 h(\beta_t)=h(p)+t\max_{x\in A}\{u_j(x)-h(p)\},
 \qquad \overline U_{\beta_t}(p)=h(p).
\]
Put $r=\max_{x\in A}\{u_j(x)-h(p)\}$ and $R_t=S(p)-f_{\beta_t}(p)$. By \cref{lem:rayerror}, $R_t\geq0$ and $R_t/t\to0$. The deviation inequality and \eqref{eq:randomsupportidentity} give
\[
 S(p)\geq C_\Pi(\beta_t;p)
       =h(p)+(1-\lambda)tr+\delta\{S(p)-R_t\}.
\]
Since $(1-\delta)S(p)=h(p)$, it follows that $(1-\lambda)r\leq\delta R_t/t$ and hence $r\leq0$. The index $j$ is arbitrary, so $u_j(x)\leq h(p)$ for every $j$ and every $x\in A$. For each such $x$,
\[
 h(p)=U(p,x)\leq M(x)\leq h(p),\qquad
 0=\sum_i p^i\{M(x)-u_i(x)\}.
\]
Nonnegativity of the summands proves \eqref{eq:contactmaxima} for every action in $A$, including actions not used by the equilibrium lottery at $p$.

By \cref{thm:randomalignment}, $A(\alpha_t)-h(\alpha_t)\to0$ almost surely, where $A=(1-\delta)S$.  The set $\mathcal Z_\Pi$ is closed in the compact simplex, because $(1-\delta)S-h$ is continuous. On this probability-one event every path has an accumulation point, and continuity puts every such point in $\mathcal Z_\Pi$. Thus $\mathcal Z_\Pi$ is also nonempty.  Under \eqref{eq:uniquevaluator} this set contains only pure worldviews.  The same probability-one event may be intersected with those on which $h(\alpha_t)\to L$, static optimality holds at every integer date, and
\[
 h(\alpha_t)-U(\alpha_t,x_{t+1})\to0.
\]
The last assertion follows from the almost-sure summability in \cref{thm:randomalignment}.  Fix a realized path in this event. Since all its accumulation points are pure, $\dist_1(\alpha_t,\mathcal P)\to0$. If the increments do not tend to zero, choose $\epsilon>0$ and infinitely many dates for which $\|\alpha_{t+1}-\alpha_t\|_1\geq\epsilon$. Compactness and finiteness of $X$ give a subsequence of these dates with
\[
 \alpha_{t_r}\to p,\qquad \alpha_{t_r+1}\to q,\qquad
 p\ne q,\qquad x_{t_r+1}=x.
\]
The limits satisfy $\|p-q\|_1\geq\epsilon$ and are pure contact states. Static optimality at $\alpha_{t_r+1}$ and continuity give $x\in z^*(q)$ and $U(q,x)=h(q)=L$. Vanishing one-period losses give $U(p,x)=L$. The contact-state result then implies $U(q,x)=M(x)$, so the two distinct pure states $p,q$ both maximize the valuation of $x$, contradicting \eqref{eq:uniquevaluator}. Hence $\|\alpha_{t+1}-\alpha_t\|_1\to0$. Eventually the path lies within $1/4$ of the pure-state set and has increments smaller than $1$. Since distinct vertices have distance $2$, it cannot switch between their $1/4$-neighborhoods thereafter. It therefore converges to one vertex, proving \eqref{eq:randompurelimit}.  The index is measurable, since
\[
 \{J_\infty=j\}=\{\lim_t\|\alpha_t-\alpha(j)\|_1=0\}
\]
on the convergence event, with any fixed index assigned on its null complement.

If a randomized equilibrium fails to purify with positive probability from some fixed initial state, \eqref{eq:uniquevaluator} cannot hold. A maximizing tie then places its payoff array in \eqref{eq:hyperplanenecessity}. This proves the empty-interior assertion.

For finite-time absorption, assume the hypotheses of \cref{thm:genericfinite}. At each $p=\alpha(j)$ with $j\in J^\dagger$, \eqref{eq:genericstrictcap} and the strict maximizing-evaluator condition bound every future assessment by $h_j$, with equality only for the pair $(p,x_j)$. By \eqref{eq:randompostpone},
\[
 0=\mathbb E_p\sum_{t=1}^{\infty}\delta^{t-1}
       \{h_j-q_p(\alpha_t,x_t)\}.
\]
The first nonnegative term is zero almost surely, so $\kappa_p(\{p\})=1$. Static uniqueness determines the action there. For $\|\alpha-p\|_1<r_j$, the score comparisons in \eqref{eq:localradius} give $g_x(\alpha)>g_y(\alpha)$ for all $y\ne x$, where $x=x_j$. The absorbing successor $p$ attains $g_x(\alpha)/(1-\delta)$, and equality in its first-period upper bound again forces $\kappa_\alpha(\{p\})=1$.

For each integer $t$, conditional on $\alpha_t$ belonging to this neighborhood, the probability that $\alpha_{t+1}\ne p$ is zero. There are finitely many such $p$ and countably many dates, so these implications hold simultaneously almost surely. The limiting pure contact state has index in $J^\dagger$ by unique static optimality and \eqref{eq:contactmaxima}. Almost-sure convergence ensures entry into its neighborhood at a finite date and absorption at the following date.
\end{proof}

\subsection{Small payoff differences and delayed changes in action}\label{sec:nearcommon}
\Cref{thm:randomcontact} rules out persistent mixing when each action has a unique maximizing worldview.  It does not give a bound, uniform over payoff perturbations, on how soon a consumer approaches purity.  The following estimate connects the common-value examples to models in which the relevant payoffs are only approximately equal.

\begin{proposition}[An approximate assessment bound]\label{prop:approximateceiling}
Fix a current worldview $\bar\alpha$ and a set of actions $D\subseteq X$.  Suppose $\epsilon\geq0$ and $g>0$ satisfy
\begin{equation}\label{eq:approximateceiling}
 q_{\bar\alpha}(\beta,x)
 \leq h(\bar\alpha)+\epsilon-g\mathbf1_{\{x\in D\}}
 \qquad\text{for all }\beta,\ x\in z^*(\beta).
\end{equation}
In every stationary randomized MPE that exists, starting from $\bar\alpha$,
\begin{equation}\label{eq:discountedbadbound}
 \mathbb E_{\bar\alpha}\sum_{t=1}^\infty
          \delta^{t-1}\mathbf1_{\{x_t\in D\}}
 \leq\frac{\epsilon}{(1-\delta)g}.
\end{equation}
If $\tau=\inf\{t\geq1:x_t\in D\}$, with $\inf\varnothing=\infty$, then for every integer $T\geq1$,
\begin{equation}\label{eq:finitehorizonprob}
 \mathbb P_{\bar\alpha}(\tau\leq T)
 \leq\min\left\{1,\frac{\epsilon}{(1-\delta)g\delta^{T-1}}\right\}.
\end{equation}
For a stationary pure-strategy MPE, the deterministic counterparts of \eqref{eq:discountedbadbound}--\eqref{eq:finitehorizonprob} also hold. If $\tau<\infty$ and $\epsilon>0$, then
\begin{equation}\label{eq:deterministictimedelay}
 \tau\geq1+\frac{\log((1-\delta)g/\epsilon)}{|\log\delta|}.
\end{equation}
\end{proposition}
\begin{proof}
The postponement inequality \eqref{eq:randompostpone}, followed by \eqref{eq:approximateceiling} and termwise integration, gives
\[
 \frac{h(\bar\alpha)}{1-\delta}
 \leq S(\bar\alpha)
 \leq\frac{h(\bar\alpha)+\epsilon}{1-\delta}
       -g\mathbb E_{\bar\alpha}\sum_{t=1}^\infty
                   \delta^{t-1}\mathbf1_{\{x_t\in D\}}.
\]
Subtracting the two constant terms proves \eqref{eq:discountedbadbound}.  On the event $\tau\leq T$, the sum includes the term $\delta^{\tau-1}\geq\delta^{T-1}$.  Hence
\[
 \delta^{T-1}\mathbf1_{\{\tau\leq T\}}
 \leq\sum_{t=1}^\infty\delta^{t-1}\mathbf1_{\{x_t\in D\}}.
\]
Taking expectations proves \eqref{eq:finitehorizonprob}.

For a stationary pure-strategy MPE, apply \cref{lem:postponement} directly to its deterministic path. With $S(\bar\alpha)=C_{\phi,z}(\phi(\bar\alpha);\bar\alpha)$, the same assessment bound gives
\[
 \frac{h(\bar\alpha)}{1-\delta}
 \leq S(\bar\alpha)
 \leq\frac{h(\bar\alpha)+\epsilon}{1-\delta}
       -g\sum_{t=1}^{\infty}\delta^{t-1}\mathbf1_{\{x_t\in D\}}.
\]
Consequently,
\[
 \sum_{t=1}^{\infty}\delta^{t-1}\mathbf1_{\{x_t\in D\}}
 \leq\frac{\epsilon}{(1-\delta)g},\qquad
 \mathbf1_{\{\tau\leq T\}}
 \leq\min\left\{1,\frac{\epsilon}{(1-\delta)g\delta^{T-1}}\right\}.
\]
This proves the deterministic assertions without a measurability requirement on the policy. If $\tau<\infty$ and $\epsilon>0$, the first inequality implies
\[
 \delta^{\tau-1}\leq\frac{\epsilon}{(1-\delta)g}.
\]
Taking logarithms and dividing by $\log\delta<0$ gives \eqref{eq:deterministictimedelay}. If $\epsilon=0$, the discounted sum is zero. Each action in $D$ is therefore excluded at every positive date, almost surely in the randomized case.
\end{proof}

\begin{corollary}[Perturbing the common-value family]\label{cor:nearcommon}
Fix the $n$-worldview payoffs $u^0$ in \eqref{eq:nworldpayoffs}, $\lambda>\lambda_c$, and $0<\delta<1$.  Set
\[
 b=\frac{\lambda}{\lambda_c}-1>0,\qquad
 d_0=\frac{(n-1)M-1}{n}>0.
\]
Let $u^\rho$ be any perturbed payoff array satisfying
\begin{equation}\label{eq:unrestrictedperturbation}
 \max_{j,x}|u_j^\rho(x)-u_j^0(x)|\leq\rho,
 \qquad 0<\rho<\min\{d_0/2,1/2\}.
\end{equation}
Write $U^\rho$, $h^\rho$, and $q^\rho$ for the utility, static value, and assessment induced by $u^\rho$. The common-value restriction need not hold for $u^\rho$. Fix a stationary pure-strategy MPE or a stationary randomized MPE of the perturbed game, and start at the equal mixture. Define $\tau=\inf\{t\geq1:x_t\ne c\}$, with $\inf\varnothing=\infty$. In the pure-strategy case the probabilities below are those of the deterministic path. For every integer $T\geq1$,
\begin{equation}\label{eq:perturbedhorizonbound}
 \mathbb P(\tau\leq T)
 \leq\min\left\{1,\frac{2\rho}{(1-\delta)b\delta^{T-1}}\right\}.
\end{equation}
Along every pure-strategy equilibrium path, and simultaneously at all dates almost surely under randomization, $x_t=c$ implies
\begin{equation}\label{eq:perturbeddistance}
 \max_i\alpha_t^i\leq\frac{M+2\rho}{1+M},\qquad
 \dist_1(\alpha_t,\mathcal P)\geq\frac{2(1-2\rho)}{1+M}>0.
\end{equation}
If the perturbed payoffs satisfy \eqref{eq:uniquevaluator}, then each pure-strategy equilibrium path converges to a pure worldview and has $\tau<\infty$; the same assertions hold almost surely under stationary randomization. For a fixed integer $T\geq1$ and $0<\eta<1$, put
\begin{equation}\label{eq:finitehorizontolerance}
 \rho_{T,\eta}
 =\min\left\{\frac{d_0}{2},\frac12,
       \frac{\eta(1-\delta)b\delta^{T-1}}{2}\right\}>0.
\end{equation}
Whenever $0<\rho<\rho_{T,\eta}$, every equilibrium specified above satisfies
\[
 \mathbb P(x_1=\cdots=x_T=c)
 =\mathbb P(\tau>T)>1-\eta.
\]
The tolerance is uniform over the perturbed payoff arrays and the equilibria satisfying these hypotheses.
\end{corollary}
\begin{proof}
At the equal mixture, the unperturbed values are $U^0(\bar\alpha,c)=0$ and $U^0(\bar\alpha,a_i)=-d_0$.  Each mixed-worldview utility changes by at most $\rho$, so $c$ remains uniquely optimal when $2\rho<d_0$ and
\[
 -\rho\leq h^\rho(\bar\alpha)\leq\rho.
\]
For every future pair, whether or not it was implementable before perturbation, the unperturbed assessment satisfies
\[
 q^0_{\bar\alpha}(\beta,c)=0,\qquad
 q^0_{\bar\alpha}(\beta,a_i)\leq-b.
\]
The two weighted components of the assessment change by at most $\rho$ in total.  Therefore
\[
 q^\rho_{\bar\alpha}(\beta,x)
 \leq\rho-b\mathbf1_{\{x\ne c\}}
 \leq h^\rho(\bar\alpha)+2\rho-b\mathbf1_{\{x\ne c\}}.
\]
Apply \cref{prop:approximateceiling} with $\epsilon=2\rho$, $g=b$, and $D=X\setminus\{c\}$ to obtain \eqref{eq:perturbedhorizonbound}.

If $c$ is optimal at $\alpha$, then for every $i$,
\[
 (1+M)\alpha^i-M-\rho
 \leq U^\rho(\alpha,a_i)
 \leq U^\rho(\alpha,c)\leq\rho.
\]
This gives the coordinate bound in \eqref{eq:perturbeddistance}.  Since
$\dist_1(\alpha,\mathcal P)=2(1-\max_i\alpha^i)$, the distance bound follows.

Under \eqref{eq:uniquevaluator}, \cref{thm:genericpurification,thm:randomcontact} give convergence to a pure worldview.  At each pure worldview its own action $a_i$ is still uniquely optimal: its payoff exceeds that of $c$ by at least $1-2\rho>0$ and every other extreme action by at least $1+M-2\rho>0$.  Continuity gives a neighborhood with the same unique optimal action.  A convergent path eventually enters such a neighborhood, so a noncompromise action is eventually chosen.  This proves $\tau<\infty$ almost surely.  For $0<\rho<\rho_{T,\eta}$, \eqref{eq:perturbedhorizonbound} gives
\[
 \mathbb P(\tau\leq T)
 \leq\frac{2\rho}{(1-\delta)b\delta^{T-1}}<\eta.
\]
Taking complements proves the final assertion.
\end{proof}

Thus breaking exact equality can change the limiting worldview while leaving finite-horizon action choices largely unchanged. The estimate bounds the probability of an early departure in equilibria that exist in the perturbed game; it does not supply a uniform upper bound on the departure time.

\section{Economic implications}\label{sec:implications}

\subsection{The action menu and the persistence of mixed worldviews}

In \eqref{eq:countermatrix}, both pure worldviews are absorbing in every stationary MPE throughout $0<\lambda<1$.  Yet when $\lambda>1/2$, every equilibrium from the balanced worldview selects the compromise action forever.  Pure-state absorption and convergence from mixed worldviews therefore describe different properties of the same model.

The initial self compares
\[
 (1-\lambda)\cdot1+\lambda\cdot(-1)=1-2\lambda
 \quad\text{with}\quad0.
\]
The first quantity evaluates an extreme action after full conversion; the second evaluates the common-value action.  At high $\lambda$, current objections to the extreme action outweigh the prospective enjoyment associated with conversion.  Persistence is possible despite unrestricted choice of the next worldview and the absence of adjustment costs.  Its source is the interaction between the action menu and the weights placed on current and future evaluations.

For the $n$-worldview family with $\lambda>\lambda_c$, the same comparison is $1-\lambda/\lambda_c$ against zero. The common action remains available as the number of perspectives increases. Fix a Borel stationary equilibrium and an initial population distribution $\nu_0$, with the same payoff array and parameters for all consumers. Let $\nu_t$ be the distribution of worldviews at date $t$ induced by that equilibrium. For deterministic policies,
\[
 \nu_t(B)=\nu_0\bigl((\phi^t)^{-1}(B)\bigr)
\]
for every Borel $B\subseteq\Delta^n$; for a randomized policy,
\[
 \nu_t(B)=\int\mathbb P_\alpha(\alpha_t\in B)\,\nu_0(d\alpha).
\]
Theorem~\ref{thm:manyworldviews} implies
\begin{equation}\label{eq:populationpersistentmass}
 \nu_t(C_c)\geq\nu_0(\mathcal T_\lambda)\qquad(t\geq0).
\end{equation}
For each initial $\alpha\in\mathcal T_\lambda$, Theorem~\ref{thm:randompersistence} gives $\mathbb P_\alpha(\alpha_t\in C_c)=1$. Thus
\[
 \nu_t(C_c)
 \geq\int_{\mathcal T_\lambda}
          \mathbb P_\alpha(\alpha_t\in C_c)\,\nu_0(d\alpha)
 =\nu_0(\mathcal T_\lambda).
\]
Deterministic equilibrium is the corresponding degenerate case. A positive initial mass on $\mathcal T_\lambda$ therefore leaves a positive mass of persistently mixed worldviews at every date.

The limiting-support theorem explains which feature of the examples matters beyond their symmetry. At a mixed accumulation point, the pure worldviews that receive positive weight must agree at the highest evaluation of every action optimal there. Under unique maximizing evaluators, every existing stationary equilibrium instead converges to purity, for every $0<\lambda,\delta<1$. Low mindset flexibility alone therefore does not distinguish persistence from eventual purity. Shared maxima are necessary for persistent mixing, but they are not sufficient: continuation incentives must also support the mixed outcome. The linear-program condition provides one such sufficient mechanism.

For unrestricted finite payoff arrays, the equality requirement confines nonpurification to a finite union of hyperplanes. A common-value outside option specifies a payoff that is independent of the available worldviews. Relative robustness concerns perturbations within this specification; unrestricted perturbations can change the long-run outcome. The two perturbation classes therefore answer different economic questions.

\subsection{Eventual purity and the timing of adjustment}
Fix $n$, $M$, $\lambda>\lambda_c$, and $\delta$, and start at the equal mixture. \Cref{cor:nearcommon} allows unrestricted perturbations of size $\rho$ around the common-value family. If the perturbed payoffs have unique maximizing evaluators, every stationary equilibrium that exists eventually approaches a pure worldview. Nevertheless, \eqref{eq:perturbedhorizonbound} bounds the probability of leaving the compromise action by any fixed date, uniformly over all those equilibria, and that bound tends to zero with $\rho$. While compromise is chosen, \eqref{eq:perturbeddistance} keeps the worldview a positive distance from purity.

The distinction concerns both behavior and timing. During an interval of compromise, the worldview weights may continue to change. The estimate concerns the first departure from action $c$, not the first date at which a pure worldview is reached. While $c$ is chosen, the distance bound in \eqref{eq:perturbeddistance} excludes purity. Conditional on equilibrium existence in nearby games, eventual convergence can therefore coexist with an arbitrarily small probability of leaving compromise during any fixed observation period. The asymptotic result alone gives no uniform upper bound on how soon a change in behavior will occur.

\subsection{Tailored worldviews and credible changes in preferences}

Under the strictly tailored condition, each implementable action has a pure worldview that uniquely selects it and maximizes its experienced utility.  When $\lambda>\lambda^\star$, that worldview is absorbing.  Selecting it is therefore a credible way to obtain a constant continuation, without requiring a commitment by future selves.  The period-by-period upper bound in \cref{thm:onestepcharacterization} is attained from the first future period, so delaying adoption cannot improve the current self's evaluation.  This explains both one-step convergence and the independence of the equilibrium set from $\delta$ on this domain.

The contrast with the happiness-dominance case of \citet[Proposition~3, pp.~730--731]{BBMZ2021} is useful.  In that case, the consumer may initially seek a happier evaluation without changing her action, while anticipating that later selves will eventually change it.  Adjustment can consequently be gradual.  In the strictly tailored class, every implementable action already has an absorbing pure destination, and the current self compares these destinations directly through $g_x(\alpha)$.

\subsection{Unique evaluation and unique implementation}
Two distinct uniqueness conditions enter the analysis. A unique maximizing evaluator of an action concerns comparisons of that same action across worldviews. It is enough for asymptotic purification in every existing stationary equilibrium, even when that evaluator would choose another action. Unique implementation concerns comparisons of different actions at a fixed pure worldview. Together with unique valuation maxima, it makes eventual adoption finite at sufficiently low mindset flexibility. Strict tailoring for every implementable action gives the still stronger one-step characterization and guarantees existence.

Weak tailoring retains the first kind of uniqueness but not the second. Accordingly, \cref{cor:weaktailoringconvergence} establishes its conditional convergence result, while \cref{thm:weaktailoringnonexistence} shows that a stationary pure-strategy equilibrium may fail to exist. In the example the difficulty is not an unwillingness to choose a pure worldview: a self can seek successors arbitrarily close to it while inducing a different tied action, but no successor attains the desired limit value. This separates the credibility of the destination's action from its ability to provide a favorable evaluation.

\subsection{Individual convergence and population outcomes}

Individual adoption of a pure worldview does not by itself imply population polarization.  For the strictly tailored model, let $\nu$ be a Borel probability measure on $\Delta^n$ describing initial worldviews among consumers who share the payoff array and $\lambda>\lambda^\star$.  Fix a stationary MPE $(\phi,z)$ with Borel measurable policies and define
\[
 B_x=\{\alpha\in\Delta^n:\phi(\alpha)=\alpha^I(x)\},
 \qquad p_x=\nu(B_x),\qquad x\in X^I.
\]
By \cref{thm:onestepcharacterization}, the sets $B_x$ form a measurable partition of $\Delta^n$.  After the first period, the worldview distribution is
\begin{equation}\label{eq:populationdistribution}
 \nu_1=\sum_{x\in X^I}p_x\,\mathsf{Dirac}_{\alpha^I(x)},
 \qquad\sum_{x\in X^I}p_x=1,
\end{equation}
where $\mathsf{Dirac}_p$ denotes the probability measure concentrated at $p$.  For every Borel set $A\subseteq\Delta^n$,
\[
 \nu_1(A)=\nu(\phi^{-1}(A))
 =\sum_{x\in X^I}\nu(B_x)\,\mathbf{1}_{\{\alpha^I(x)\in A\}},
\]
which proves \eqref{eq:populationdistribution}.  Absorption gives $\nu_t=\nu_1$ for all $t\geq1$.  The population reaches a common pure worldview if one $p_x$ equals one; at least two distinct pure worldviews persist if at least two weights are positive.  The one-step result alone does not determine which of these population outcomes occurs.  For $n\geq2$, if $\nu$ is absolutely continuous with respect to $(n-1)$-dimensional Lebesgue measure on the affine hull of $\Delta^n$, \cref{cor:policycells} implies that these weights are unaffected by tie-breaking. The case $n=1$ has a single worldview and a unique worldview policy. This aggregation keeps the individual decision problem fixed and introduces no social interaction.

The comparisons throughout use the model's cardinal evaluations.  The monotonicity of $H_t$ concerns the discounted average of the utilities experienced along a single equilibrium path.  It is not a welfare comparison between policies that induce different worldviews, nor a social ranking of compromise and conversion.  Welfare with endogenous preferences is a separate question raised in \citet[Section~V, p.~751]{BBMZ2021}.

\section{Conclusion}

The ability to choose preferences need not lead to a pure worldview. At low mindset flexibility, a common-value action can sustain mixed worldviews in every stationary equilibrium from a relatively open set of initial states, even when every pure worldview is absorbing. The general restriction on limiting behavior is equally important: every pure worldview in the support of a mixed accumulation point must maximize the valuation of every action optimal there. If each implementable action has a unique maximizing evaluator, every stationary equilibrium that exists converges to a pure worldview for all $0<\lambda,\delta<1$. Experienced utility converges without uniqueness. These results extend to stationary randomization.

The limiting outcome and the timing of adjustment are distinct. Payoff perturbations that remove shared maxima can restore eventual purity while leaving the probability of departing from compromise during a fixed horizon arbitrarily small. Strict tailoring addresses the attainability of a favorable continuation: at sufficiently low mindset flexibility, it ensures existence and a complete one-step equilibrium characterization. Under weak implementation, a stationary pure-strategy equilibrium may instead fail to exist because a continuation optimum is not attained.

These are conclusions about individual preference formation. Population outcomes also depend on the initial distribution of worldviews, and welfare comparisons require evaluating policies that change preferences. The results do not characterize equilibrium existence for arbitrary finite menus or limiting behavior in every tied-payoff environment.

\appendix
\section{Boundary values of mindset flexibility}\label{sec:boundary}

Fix $0<\delta<1$ and retain the stationary pure-strategy equilibrium concept.  The postponement identity in \cref{lem:postponement} remains valid at $\lambda=0$ and $\lambda=1$, since its proof uses neither $\lambda>0$ nor $\lambda<1$.

\begin{proposition}[Perfect mindset flexibility]\label{prop:lambda0}
Suppose $\lambda=0$, and put $M=\max_{\alpha\in\Delta^n}h(\alpha)=\max_{j,x}u_j(x)$.  In every stationary MPE, from every initial worldview,
\[
 h(\alpha_t)=M\quad\text{for all }t\geq1.
\]
If $h$ has a unique maximizer $p$, then $p$ is pure and every stationary MPE selects $p$ after one period.  A stationary MPE exists at $\lambda=0$ without the uniqueness assumption.
\end{proposition}
\begin{proof}
For every $\alpha\in\Delta^n$ and $x\in X$,
\[
 U(\alpha,x)=\sum_j\alpha^j u_j(x)\leq\max_{j,y}u_j(y)=M.
\]
A maximizing pair $(j,x)$ gives a pure worldview $p=\pure{j}$ with $h(p)=M$.  In an arbitrary stationary MPE, set $p_0=p$ and $p_{k+1}=\phi(p_k)$.  The universal upper bound and \cref{lem:postponement} give
\[
 \frac{M}{1-\delta}\geq C_{\phi,z}(\phi(p);p)
 \geq\frac{h(p)}{1-\delta}=\frac{M}{1-\delta}.
\]
Since $\lambda=0$, equality means
\[
 0=\sum_{k=1}^{\infty}\delta^{k-1}[M-h(p_k)],
\]
and all brackets are nonnegative.  Hence $h(p_k)=M$ for every $k\geq1$, as well as for $k=0$.  From any current worldview $\alpha$, choosing $p$ as successor consequently gives
\[
 C_{\phi,z}(p;\alpha)
 =\sum_{k=0}^{\infty}\delta^k h(p_k)=\frac{M}{1-\delta}.
\]
For the equilibrium continuation from $\alpha$, optimality and the same upper bound imply
\[
 \frac{M}{1-\delta}\leq C_{\phi,z}(\phi(\alpha);\alpha)
 \leq\frac{M}{1-\delta}.
\]
Expanding the difference once more gives
\[
 0=\sum_{k=0}^{\infty}\delta^k[M-h(\phi^{k+1}(\alpha))],
\]
so $h(\phi^{k+1}(\alpha))=M$ at every future date.  If $h$ has a unique maximizer, it must be the pure maximizer $p$ already exhibited, and all these states equal $p$.

For existence without uniqueness, choose a pure maximizer $p$, set $\phi(\alpha)=p$ for every $\alpha$, and take any $z(\alpha)\in z^*(\alpha)$.  The current-action condition holds by construction.  For every candidate successor $\beta$, the worldview condition follows from
\[
 C_{\phi,z}(\beta;\alpha)
 =h(\beta)+\frac{\delta M}{1-\delta}
 \leq\frac{M}{1-\delta}=C_{\phi,z}(p;\alpha).\qedhere
\]
\end{proof}

\begin{proposition}[Perfect mindset inflexibility]\label{prop:lambda1}
Suppose $\lambda=1$.  For any selection $z(\alpha)\in z^*(\alpha)$, the policy $\phi(\alpha)=\alpha$ is a stationary MPE.  In particular, every initial mixed worldview can remain unchanged forever.
\end{proposition}
\begin{proof}
Under the identity policy, choosing next worldview $\beta$ induces the action $z(\beta)$ forever.  Its evaluation under the current worldview $\alpha$ is
\[
 C_{\phi,z}(\beta;\alpha)
 =\sum_{k=0}^{\infty}\delta^k U(\alpha,z(\beta))
 =\frac{U(\alpha,z(\beta))}{1-\delta}
 \leq\frac{h(\alpha)}{1-\delta}
 =C_{\phi,z}(\alpha;\alpha).
\]
Hence remaining at $\alpha$ is optimal at every state.
\end{proof}

\end{document}